\pdfoutput=1
\documentclass[10pt,a4paper]{article}
\usepackage[utf8]{inputenc}
\usepackage[english]{babel}
\usepackage{amsmath}
\usepackage{amsfonts}
\usepackage{amssymb}
\usepackage{amsthm}
\newtheorem{thm}{Theorem}[section]
\usepackage{graphicx}
\usepackage{caption}
\usepackage[font=footnotesize,labelfont=bf]{subcaption}

\usepackage{lmodern}
\usepackage{mathrsfs}
\usepackage{filecontents}
\usepackage{epstopdf}

\theoremstyle{definition}
\newtheorem{definition}[thm]{Definition}

\newtheorem{lem}[thm]{Lemma}

\usepackage{enumerate}

\makeatletter
\long\def\@ympar#1{%
  \@savemarbox\@marbox{\scriptsize #1}%
  \global\setbox\@currbox\copy\@marbox
  \@xympar}
\makeatother
\usepackage{comment}
\usepackage{varwidth}
\usepackage{xcolor}

\begin{document}
\title{On the propagation of singularities for a class of\\ linearised hybrid inverse problems}
\title{Propagation of singularities for linearised hybrid data impedance tomography}
\author{Guillaume Bal, Columbia University\\ Kristoffer Hoffmann, Technical University of Denmark\\  Kim Knudsen, Technical University of Denmark\\}
\renewcommand{\today}{}
\maketitle
\let\oldfbox\fbox
\renewcommand{\fbox}[1]{\vspace{0.3cm}\noindent\\\oldfbox{#1}}
\renewcommand{\d}[0]{\ensuremath{\operatorname{d}}}
\begin{abstract}
For a general formulation of linearised hybrid inverse problems in impedance tomography, the qualitative properties of the solutions are analysed. Using an appropriate scalar pseudo-differential formulation, the problems are shown to permit propagating singularities under certain non-elliptic conditions, and the associated directions of propagation are precisely identified relative to the directions in which ellipticity is lost. The same result is found in the setting for the corresponding normal formulation of the scalar pseudo-differential equations. A numerical reconstruction procedure based of the least squares finite element method is derived, and a series of numerical experiments visualise exactly how the loss of ellipticity manifests itself as propagating singularities.
\end{abstract}
\section{Introduction}
Modern tomographic methods typically involve a trade-off in performance. For example, in computerized tomography the subject's exposure to potentially harmful ionizing photons is limited by acquiring  smaller amounts of data but as a consequence artifacts appear in reconstructions. Another exmaple is Electrical Impedance Tomography (EIT), which is harmless and where the contrast in some situations can be excellent while the image resolution is very low. To overcome such limitations, new tomographic paradigms taking advantage of several different physical fields are emerging. A common feature of these methods is that a physical coupling  makes it possible, at least in theory, to obtain very good interior reconstructions with both high resolution and high contrast  by just a few simple boundary measurements. This particular field of tomography is often referred to as hybrid data or coupled physics tomography.

The mathematical description of a hybrid data tomography  can be considered as a two-step process: The first step is related to the recovery of the interior data, deals with the modelling of the experimental apparatus, the physical fields and couplings, and provides a strategy on how to calculate the interior data. 
The second step is the reconstruction of the relevant physical parameter(s) using the available mathematical models, the exterior measurements and the recovered interior data sets. In this paper we will consider a mathematical and numerical analysis of a general class of non-linear hybrid inverse problems in impedance tomography augmented with interior data of infinite precision and therefore we only consider the second step.

Hybrid data problems in impedance tomography are expected to produce images of high resolution and high contrast, but unfortunately in some situations artifacts limit the image quality. Figure~\ref{fig:artifacts} shows a numerical phantom and a reconstruction with severe artifacts obtained from solving a linearized problem; reconstructions with similar artifacts can be found in litterature.
\begin{figure}[!htb]
        \centering
        \begin{subfigure}[b]{0.22\textwidth}
                \centering
                \includegraphics[width=\textwidth]{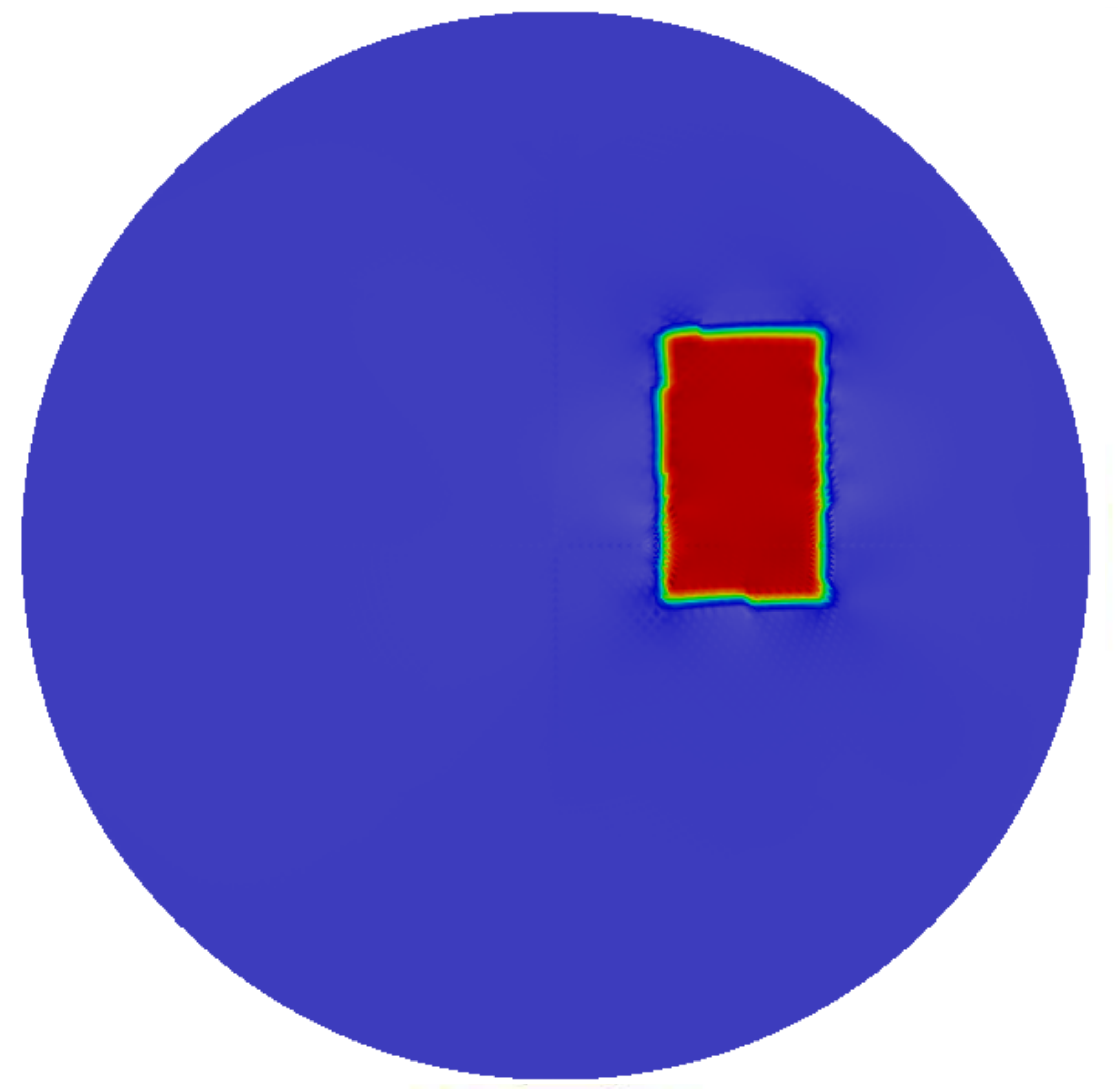}
                \caption{Phantom}
        \end{subfigure}
           \begin{subfigure}[b]{0.215\textwidth}
                \centering
                \includegraphics[width=\textwidth]{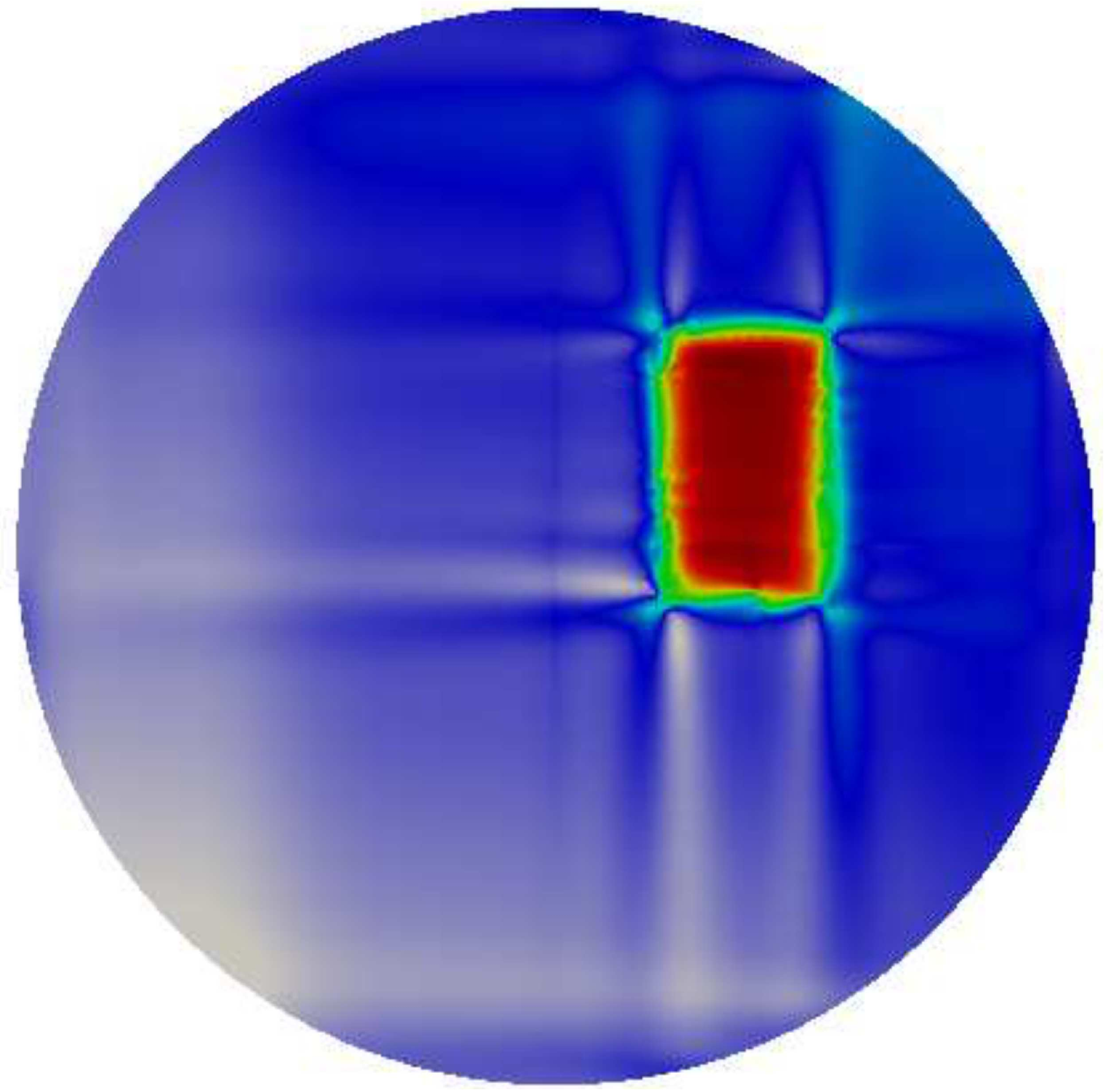} \captionsetup{justification=centering}
                \caption{Reconstruction}
        \end{subfigure}
                   \begin{subfigure}[b]{0.038\textwidth}
                \centering
                \includegraphics[width=\textwidth]{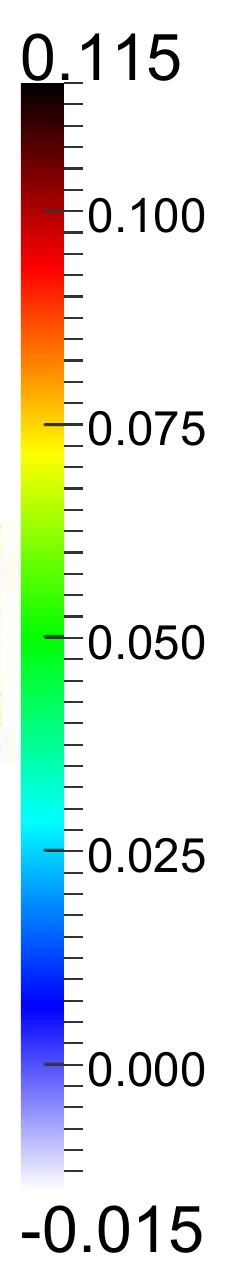}
             \caption*{}
        \end{subfigure}
        \caption{Artifacts in reconstruction due to propagation of singularities.}\label{fig:artifacts}
\end{figure}
In Computerized Tomography and Photoacoustic Tomography artifacts due to limitations in the data are well understood (see e.g.~ \cite{Frikel2013,FrikelQuinto2013, FrikelQuinto2015}), but much less is known for hybrid data problems in impedance tomography. The goal of this paper is to investigate such artifacts both theoretically using micro-local analysis and numerically. We will characterize possible artfacts and also see how the defects can be repaired. 

Hybrid inverse problems in impedance tomography are described by the
following mathematical model: Let $\Omega$ be an open, smooth and
bounded subset of $\mathbb{R}^n (n=2,3)$ and let $\sigma$ be a scalar
function bounded from above and below by positive constants in
$\Omega$. Consider the set of $J$ generalized Laplace problems
\begin{gather}\label{eq:condprob}
\left\{\begin{aligned}
\nabla \cdot (\sigma \nabla u_j) &=0 &&\text{in } \Omega,\\
u_j &=f_j &&\text{on } \partial \Omega,
\end{aligned}
\right. \quad 1\leq j \leq J.
\end{gather}
Uniqueness and regularity of solutions to such elliptic problems have been thoroughly analysed in the literature (see for instance~\cite{evans2010partial} or~\cite{gilbarg1977elliptic}).
Assume that additional scalar interior data expressed by non-linear functionals of $\sigma$ and $\nabla u_j$ are available. This text will consider interior data of the type
\begin{gather}\label{eq:data}
H_j = \sigma \vert \nabla u_j \vert^p\,\, \text{ in } \Omega, \quad p>0, \quad 1\leq j \leq J.
\end{gather}
The goal is to use the knowledge of the interior data $\{H_j\}_{j=1}^J$ to reconstruct the scalar function~$\sigma$. For $p=1$ this setting corresponds to the mathematical description of a novel imaging method called Current Density Impedance Imaging, where the data functionals $\{H_j\}_{j=1}^J$ model the interior current densities. For $p=2$ it corresponds to the mathematical description of Ultrasound Modulated Electrical Impedance Tomography (UMEIT), also known as Acousto-Electric Tomography, and $\{H_j\}_{j=1}^J$  model the interior power densities. In both cases $\sigma$ represents the scalar electrical conductivity and $u$ the electric potential. A survey of these and other hybrid imaging methods can be found in the review papers on the subject~\cite{ammari2008introduction,bal2012hybridreview,kuchment2012mathematics}. 

We will analyse the problem in the general case $p>0$ from a purely mathematical perspective and not limit the analysis to the values of $p$ which correspond to a mathematical model of certain physical quantities. However, from a practical perspective, the cases $p=1$ and $p=2$ seem to be the most interesting.

For the presented hybrid inverse problem, an analysis of the corresponding linearisation has previously been done by multiple authors. Some use methods from differential and pseudo-differential calculus to analyse the ellipticity of the linearised inverse problem and state the related stability properties when the equations are augmented with boundary conditions that satisfy the Lopatinskii criterion~\cite{bal2012hybrid,kuchment2012stabilizing}. However, one main problem is that it can be impractical or impossible to obtain an elliptic system. This could be the case if the number of measurements and choice of boundary potentials are restricted, or if there are access restrictions to parts of the boundary. The work presented in this paper is motivated by the question: What happens to the reconstruction if the linearised problem is not elliptic?  Using microlocal analysis, a theorem on the propagation of singularities can partly answer this question.

In this paper we first show that the linearised inverse problem, expressed by scalar pseudo-differential equations, in certain situations shows a hyperbolic nature, where the loss of ellipticity manifest itself as the propagation of singularities in specific directions. We show how these directions are completely determined by the directions in which ellipticity is lost. The main theoretical novelty of this paper is presented in Sec.~\ref{sec:propnormaloperator}, where the corresponding least squares formulation is analysed and we show a similar result based on the theory for pseudo-differential operators with principal symbols of constant multiplicity.

The theoretical results are visualised in two dimensions using a numerical implementation based on the least squares finite element method. This provides a robust reconstruction framework which is formulated independently of the number of measurements. The presented numerical results provide evidence that singularities indeed propagate in the predicted directions and give valuable information on how to avoid unnecessary artifacts in the regions of interest in situations where the linearised problem cannot be guaranteed to be elliptic. 

The rest of this paper is organised as follows: In Sec.~2 we formulate the linearised inverse problem in the form of a system of PDEs. A microlocal analysis of the linearised inverse problem is presented in Sec.~3. We express the inverse problem in a scalar pseudo-differential formulation, and in the case of multiple measurements, we present the corresponding normal equation. The principal symbol of the relevant operators are analysed, and we find the bicharacteristic curves along which singularities can propagate. In Sec.~4 we present a numerically implementation of the linearised inverse problem based on the least squares finite element method. Results from this implementation are presented in Sec.~5, where it is visualised how the loss of ellipticity indeed manifest itself as propagation singularities exactly in the directions predicted by theory. A conclusion based of these findings is made in Sec.~6.
\section{Setting the stage}
\label{chap:mathset}
In this section we derive the linearised inverse problem and present it in a matrix form. Next we tranform it into a scalar form that will be the starting point for the classification of the relevant operators and the analysis of the propagation of singularities. 

\subsection{The linearized problem}
Let $C_+^\infty(\bar{\Omega})$ be the set of smooth functions in $\bar{\Omega}$ with a strictly positive lower and upper bound. In the rest of this paper we assume that
 $\tilde{\sigma} \in C_+^\infty(\bar{\Omega})$ and that the boundray conditions $f_j \in C^\infty (\partial\Omega).$ Let $\tilde{u}_j$ be the unique reference solution to~\eqref{eq:condprob} when $\sigma$ is replaced by $\tilde{\sigma}$. By standard elliptic regularity theory it follows that $\tilde{u}_j \in C^\infty(\bar\Omega)$~\cite{evans2010partial}. We will assume throughout that $|\nabla \tilde u_j| \not = 0$ in $\bar \Omega.$

The  Fréchet derivative of  the interior data operator $\mathcal{H}_j: \sigma \mapsto \sigma \vert \nabla u_j \vert^p$ at $\tilde{\sigma}$ is under the stated  assumptions 
given by 
\begin{gather}
\d \mathcal{H}_j\vert_{\tilde{\sigma}} [\delta \sigma] = \delta \sigma \vert \nabla \tilde{u}_j \vert^p + p \tilde{\sigma} \frac{\nabla \tilde{u}_j \cdot \nabla \delta u_j}{\vert \nabla \tilde{u}_j \vert^{2-p}}\, \text{ in } \Omega, \quad 1\leq j \leq J,
\end{gather}
where $\delta u_j \in H^1_0(\Omega)$ solves
\begin{gather}\label{eq:deltau}
\left\{\begin{aligned}
\nabla \cdot (\tilde{\sigma} \nabla \delta u_j)&=-\nabla \cdot (\delta \sigma \nabla \tilde{u}_j)\\
 \delta u_j &= 0
\end{aligned}
\begin{aligned}
&\text{ in } \Omega,\\
&\text{ on } \partial \Omega,
\end{aligned}
\right. \quad 1\leq j \leq J.
\end{gather}
The linearized problem for the difference  $\sigma - \tilde{\sigma}=\delta \sigma$ is thus
\begin{gather}\label{eq:lindata}
\begin{aligned}
\d \mathcal{H}_j\vert_{\tilde{\sigma}} [\delta \sigma] &= H_j-\tilde{H}_j \\
\end{aligned}
\begin{aligned}
&\text{ in } \Omega,\\
\end{aligned}
\quad 1\leq j \leq J,
\end{gather}
where the reference interior data is given by $\tilde{H}_j = \tilde{\sigma} \vert \nabla \tilde{u}_j \vert^p$. 

 The equations~\eqref{eq:deltau} and~\eqref{eq:lindata} form a collection of linear PDE problems for $\{\delta \sigma, \{\delta u_j\}_{j=1}^J\}$ expressed in the matrix form
\begin{gather}\label{eq:linsystem}
\begin{bmatrix}
\vert \nabla \tilde{u}_j \vert^p & p \tilde{\sigma} \frac{\nabla \tilde{u}_j \cdot \nabla [\cdot]}{\vert \nabla \tilde{u}_j \vert^{2-p}}\\
\nabla \cdot ([\cdot] \nabla \tilde{u}_j) & \nabla \cdot (\tilde{\sigma}  \nabla [\cdot])
\end{bmatrix} \begin{bmatrix}
\delta \sigma \\
\delta u_j
\end{bmatrix} = \begin{bmatrix}
H_j - \tilde{H}_j \\
0
\end{bmatrix} \text{ in } \Omega, \quad 1 \leq j \leq J,
\end{gather}
equipped with homogeneous Dirichlet boundary conditions for $\{\delta u_j\}_{j=1}^J$. This is the matrix formulation of the linearised inverse problem.

\subsection{Transformation into a system of scalar problems}
The symbolic calculus becomes more transparant if we express the linearised inverse problem by a one scalar equations. 
Elimination of $\{\delta u_j\}_{j=1}^J$~\eqref{eq:linsystem} gives for each $j$ the scalar 
equation
\begin{gather}\label{eq:scalarprobs}
P_j \delta \sigma = H_j-\tilde{H}_j, \quad 1 \leq j \leq J,
\end{gather}
where $P_j$ is the  non-local 
operator defined by
\begin{gather}\label{eq:pj}
P_jh = \vert \nabla \tilde{u}_j \vert^ph +  p \tilde{\sigma} \frac{\nabla \tilde{u}_j \cdot \nabla L_{\tilde{\sigma}}^{-1}(\nabla \cdot (h  \nabla \tilde{u}_j))}{\vert \nabla \tilde{u}_j \vert^{2-p}}.
\end{gather}
Here $L_{\tilde{\sigma}}^{-1} := \left(-\nabla \cdot(\tilde{\sigma} \nabla[\cdot]) \right)^{-1}$ denotes the solution operator to
\begin{gather}\label{eq:deltau}
\left\{\begin{aligned}
-\nabla \cdot (\tilde{\sigma} \nabla v)&=f\\
 v &= 0
\end{aligned}
\begin{aligned}
&\text{ in } \Omega,\\
&\text{ on } \partial \Omega,
\end{aligned}
\right. . 
\end{gather}

The linear scalar equations~\eqref{eq:scalarprobs} are for $J>1$ formally an overdetermined system of equations for $\delta \sigma$. In this case we are going to analyse the normal form of~\eqref{eq:scalarprobs} given by the single scalar equation
\begin{gather}\label{eq:lsform}
\sum_{j=1}^J P_j^*P_j \delta \sigma = \sum_{j=1}^J P_j^*(H_j-\tilde{H}_j).
\end{gather}
Here $P_j^*\in  \Psi^0(\Omega)$ denotes the formal adjoint of $P_j$~\cite[p. 21]{Treves1}.


\section{Microlocal analysis of the linearised inverse problem}
As mentioned in the introduction, several authors have analysed the general stability properties of the linearised inverse problem \eqref{eq:linsystem}. Kuchment and Steinhauer~\cite{kuchment2012stabilizing}, and Bal~\cite{bal2012hybrid} have shown how this depends on the set of gradients $\{\nabla \tilde{u}_j\}_{j=1}^J$. In this section we will extend these results to the normal formulation~\eqref{eq:lsform} and explain exactly how the loss of stability manifest itself as propagating singularities in the reconstruction. It should be noted that the first two theorems in this section are  similar to results by Kuchment and Steinhauer~\cite{kuchment2012stabilizing} and included here for completeness.

In the following we denote by $\Psi^m(\Omega)$ the space of standard pseudo-differential operators of order $m$ in $\Omega$~\cite[Def. 2.3]{Treves1} and by $S^m(\Omega)$ the space of symbols of order $m$~\cite[Def. 4.1]{Treves1}. For a classical pseudo-differential operator $P \in \Psi^m(\Omega)$ we denote by $\mathfrak{p}$ the full symbol, $\mathfrak{p}^m$ the principal symbol, by $\mathfrak{p}^{m-1}$ the part of the symbol which is positively homogeneous of order $m-1$ in $\xi$ and so on. 

\subsection{Classification of the scalar operator and the normal operator}\label{sec:classifi}
The first theorem concerns the operator $P_j:$
\begin{thm}\label{thm:pj}
Let $\tilde{\sigma}\in C_+^\infty(\bar{\Omega})$ and $\partial \Omega$ be smooth. Then the operator $P_j$ given by~\eqref{eq:pj} is a classical and properly supported pseudo-differential operator in $\Omega$ of order zero, with principal symbol
\begin{gather}\label{eq:pjprinsym}
\mathfrak{p}_j^0 := \vert \nabla \tilde{u}_j \vert^p\left(1   - p \frac{ (\nabla \tilde{u}_j \cdot \xi)^2}{\vert \nabla \tilde{u}_j \vert^{2} |\xi| ^2} \right).
\end{gather}
In addition the normal operator $P_j^*P_j$ is also a classical and properly supported pseudo-differential operator in $\Omega$ of order zero with principal symbol $(\mathfrak{p}_j^0)^2.$
\end{thm}
\begin{proof}
It follows from standard elliptic regularity that $\tilde{\sigma} \in  C_+^\infty(\bar{\Omega})$ implies $\tilde{u}_j  \in C^\infty(\Omega)$. Therefore $\nabla \cdot [\cdot]\nabla \tilde{u}_j$ and  $\nabla \tilde{u}_j \cdot \nabla[\cdot]$ are first order differential operators with smooth coefficients and as a result both classical and properly supported pseudo-differential operators in $\Psi^1(\Omega)$~\cite[p. 11]{komech1999elements},\cite[p. 16]{shubin2001pseudodifferential}. Because $\nabla \cdot \tilde{\sigma} \nabla [\cdot]$ is a second order elliptic differential operator, it follows that there exists a parametrix, denoted by $L_{\tilde{\sigma}}^{-1}$, which is a classical and properly supported pseudo-differential elliptic operator and an element of $\Psi^{-2}(\Omega)$~\cite[Thm. 1.3]{komech1999elements}. The operator
\begin{gather}
\nabla \tilde{u}_j \cdot \nabla L_{\tilde{\sigma}}^{-1}(\nabla \cdot ([\cdot]  \nabla \tilde{u}_j))
\end{gather}
is therefore a composition of pseudo-differential operators that are both classical and properly supported in $\Omega$ and therefore also such an operator. It follows that $P_j\in \Psi^0(\Omega)$ is also both classical and properly supported. The principal symbol of $P_j$ is then the product of the principal symbols of each operator in the composition and this is easily found to be the expression given by~\eqref{eq:pjprinsym}.

Note that since $P_j$ is classical and properly supported the composition $P_j^*P_j$ makes sense and the sympol is obtained by squaring $\mathfrak{p}_j^0 .$ 
\end{proof}

The microlocal analysis is related to the properties of the principal symbol of $P_j$. As we now show, the ellipticity of $P_j$ depends on the parameter $p$.
\begin{thm}\label{thm:pjell}
Suppose $\vert \nabla \tilde{u}_j \vert \neq 0$. Then $P_j$ is elliptic in $\Omega$ if, and only if, $p<1$.
\end{thm}
\begin{proof}
$P_j$ is elliptic in $\Omega$ if, and only if, its principal symbol is non-zero for all $(x,\xi)\in T^*\Omega\setminus 0$. When $\vert \nabla \tilde{u}_j \vert \neq 0$ a vanishing principal symbol is equivalent to
\begin{gather}
\frac{(\nabla \tilde{u}_j(x)\cdot \xi)^2}{\vert \nabla \tilde{u}_j(x) \vert^2 \vert \xi \vert ^2} = \frac{1}{p},
\end{gather}
for some $(x,\xi)\in T^*\Omega\setminus 0$. The expression of the left hand side is clearly bounded by 1. Thus if $p<1$ the equality is nowhere satisfied and $P_j$ is therefore elliptic in $\Omega$. If $p\geq 1$, consider $\xi =v(x)$ where $v(x)$ is the vector $\nabla \tilde{u}_j(x)$ rotated by an angle $\alpha_j$, satisfying $\cos\alpha_j = \frac{1}{\sqrt{p}}$. Then the principal symbol vanishes for all $(x,v(x))\in T^*\Omega\setminus 0$ and $P_j$ is therefore nowhere elliptic.
\end{proof}
For $p\geq 1$ the operaaor $P_1$ is not elliptic and the calculation shows the relevance of one particular direction $\xi$ given by a rotation of the field $\tilde \nabla u(x)$ by the angle $\alpha_j$ satisfying $\cos(\alpha_j) = \frac 1 {p_j}.$ We will say that $P_j$ loses ellipticity in the direction $\xi$ at the point $x.$

A similar result can be found by~Bal~\cite{bal2012hybrid}. The theorem above implies that if $p\geq1$, more than one measurement of the interior data is necessary to have at least one elliptic operator in the set $\{P_j\}_{j=1}^J$ at every point $(x,\xi) \in T^*\Omega\setminus 0$. In this situation it makes sense to consider the problem formulated in the scalar normal form expressed by~\eqref{eq:lsform}. The condition that ensures ellipticity of the operator associated with the normal equation is stated in the following theorem.
\begin{thm}\label{thm:pj2ell}
$\sum_{j=1}^J P_j^* P_j$ is elliptic at $(x,\xi)\in T^*\Omega\setminus 0$ if, and only if, there exists $\vert \nabla \tilde{u}_j(x) \vert \neq 0$ such that $\frac{(\nabla \tilde{u}_j(x)\cdot \xi)^2}{\vert \nabla \tilde{u}_j(x) \vert^2 \vert \xi \vert ^2} \neq \frac{1}{p}$.
\end{thm}
\begin{proof}
The principal symbol of $\sum_{j=1}^J P_j^* P_j$ is 
given by
\begin{gather}\label{eq:prinsymls}
\sum_{j=1}^J \vert \nabla \tilde{u}_j \vert^{2p}\left(  1 - p \frac{ (\nabla \tilde{u}_j \cdot \xi)^2}{\vert \nabla \tilde{u}_j \vert^2 |\xi| ^2} \right)^2.
\end{gather}
This expression is clearly positive at $(x,\xi)\in T^*\Omega\setminus 0$ if, and only if, there exists $\vert \nabla \tilde{u}_j(x) \vert \neq 0$ such that $\frac{(\nabla \tilde{u}_j(x)\cdot \xi)^2}{\vert \nabla \tilde{u}_j(x) \vert^2 \vert \xi \vert ^2} \neq \frac{1}{p}$.
\end{proof}
The implication of Thm.~\ref{thm:pj2ell} is that the ellipticity of $\sum_{j=1}^J P_j^* P_j$ is determined by the set $\{\nabla \tilde{u}_j \}_{j=1}^J$. A simple geometrical analysis gives the following results~\cite{bal2012hybrid,kuchment2012stabilizing}: Loss of ellipticity corresponds to the intersection of the cones (in the $\xi$-variable) given by the equations $\frac{(\nabla \tilde{u}_j(x)\cdot \xi)^2}{\vert \nabla \tilde{u}_j(x) \vert^2 \vert \xi \vert ^2} - \frac{1}{p} = 0$ at some point away from the apex. Thus, if $p<1$ the operator is always elliptic. For $p\geq1$, choosing $J\geq n$ is necessary to obtain an elliptic operator. Now, if $\nabla \tilde{u}_1$ and $\nabla \tilde{u}_2$, corresponding to the imposed boundary conditions $f_1$ and $f_2$, are known to be nowhere parallel, then making a third measurement with the boundary condition $f_1+f_2$ ensures that the system becomes elliptic in both two and three dimensions. Actually this result can be extended to any dimension larger than or equal two~\cite{bal2012hybrid}. However, it is in general not possible to ensure that the gradients are not parallel. In two dimensions it is possible to ensure that the gradients are non-zero and never parallel if $\tilde{\sigma}$ is sufficiently smooth~\cite{alessandrini2001univalent}. If furthermore the boundary conditions are almost two-to-one, then the magnitude of the gradients are bounded below by a positive constant in $\Omega$~\cite{nachman2007conductivity}. In three dimensions it can be shown that similar results cannot exist, but in certain situations critical points can be avoided by a certain mathematical construction, corresponding to choosing boundary conditions as traces of specific complex geometric optics (CGO) solutions~\cite{bal2012cauchy}.
 \subsection{Propagation of singularities}
When a problem is not elliptic it is possible for the solution to have more singularities than those present in the right-hand side of the equation. However, the additional singularities can only be present along curves where the principal symbol vanishes. This is a simple interpretation of the classical result on the propagation of singularities. As we will see in the following analysis, it is possible to calculate these curves for the linearised inverse problem and state necessary conditions for the presence of singularities along such curves.
\subsubsection{The operator $\boldsymbol{P_1}$}
When only a single interior data set is available, the relevant equation to analyse is
\begin{gather}\label{eq:p1prob}
P_1 \delta \sigma = H_1-\tilde{H}_1.
\end{gather}
Thus, the propagation of singularities depends on the (principal) symbol of $P_1,$ which by Thm.~\ref{thm:pj} is given by
\begin{equation}\label{eq:principalP1}
\mathfrak{p}_1^0 := \vert \nabla \tilde{u}_1 \vert^p \left(  1 - p \frac{ (\nabla \tilde{u}_1 \cdot \xi)^2}{\vert \nabla \tilde{u}_1 \vert^{2} |\xi| ^2} \right).
\end{equation}
The standard theory on propagation of singularities concerns principal symbol of real principal type. By definition \cite{egorov1993microlocal}, a scalar symbol $\mathfrak{q}(x,\xi)$ is of real principal type if, and only if, it is real and for $\xi \not = 0$
\begin{gather}
\mathfrak{q}(x,\xi) = 0 \Rightarrow \frac{\partial}{\partial \xi} q(x,\xi) \neq 0~.
\end{gather}
We will next show that the symbol \eqref{eq:principalP1} is of this kind:
\begin{lem}\label{lem:qrpt}
For $p\neq1$, the principal symbol \eqref{eq:principalP1} of $P_1$ is of real principal type.
\end{lem}
\begin{proof}
Elliptic operators are trivially of principal type, so it suffices to consider the case $p\geq1$. Clearly $\mathfrak{p}_1^0$ is real and we find that
\begin{gather}\label{eq:nonparallel}
\mathfrak{p}_1^0(x,\xi) = 0 \Leftrightarrow \frac{(\nabla \tilde{u}_1\cdot \xi)^2}{\vert \nabla \tilde{u}_1 \vert^2 \vert \xi \vert ^2} = \frac{1}{p},
\end{gather}
and therefore
\begin{gather}
\left.\frac{\partial}{\partial \xi} \mathfrak{p}_1^0(x,\xi) \right\vert_{\mathfrak{p}_1^0=0} = 2p\frac{\vert \nabla \tilde{u}_1 \vert^p}{|\xi|^2} \left(\frac{1}{p} \xi -\frac{\nabla \tilde{u}_1 \cdot \xi}{\vert \nabla \tilde{u}_1 \vert^2} \nabla \tilde{u}_1\right).
\end{gather}
This means that $\mathfrak{p}_1^0$ is of real principal type if, and only if, $\mathfrak{p}_1^0(x,\xi) = 0$ implies
\begin{equation}\label{eq:inequalP1}
\frac{1}{p} \xi -\frac{\nabla \tilde{u}_1 \cdot \xi}{\vert \nabla \tilde{u}_1 \vert^2} \nabla \tilde{u}_1 \neq 0.
\end{equation}
When $p=1$,  the vectors $\xi$ and $\nabla \tilde{u}_j$ are parallel when \eqref{eq:nonparallel} holds and one finds that \eqref{eq:inequalP1} cannot be satisfied. When $p> 1$, the non-zero vectors $\nabla \tilde{u}_1$ and $\xi$ cannot be parallel, and this is sufficient to conclude that the inequality is satisfied.
\end{proof}
The classical theorem on propagation of singularities for pseudo-differential operators of real principal type, which is due to Duistermaat and Hörmander~\cite{duistermaat1972fourier}, explains how singularities propagate along the bicharacteristics. In the following theorem we show that the bicharacteristic curves are perpendicular to the direction in which ellipticity is lost.
\begin{thm}\label{thm:aj1propbi}
For $p>1$, the bicharacteristic curves of $P_1$ are perpendicular to the directions in which ellipticity is lost.  
\end{thm}
\begin{proof}
The bicharacteristic strips $(x(t),\xi(t))\in T^*\Omega\setminus 0, t \in \mathbb{R}$ are integral curves of the system of equations
\begin{equation}
\left\{\begin{aligned}
 \frac{dx(t)}{dt} &= \frac{\partial}{\partial \xi} \mathfrak{p}_1^0(x(t)),\xi(t)), \\
\frac{d\xi(t)}{dt} &= -\frac{\partial}{\partial x} \mathfrak{p}_1^0(x(t)),\xi(t)),
\end{aligned}\right.  
\end{equation}
where $\mathfrak{p}_1^0(x(t)),\xi(t))=0$~\cite{egorov1993microlocal}. We find that
\begin{gather}
\frac{dx(t)}{dt} = 2p\frac{\vert \nabla \tilde{u}_1 \vert^p}{|\xi(t)|^2} \left(\frac{(\nabla \tilde{u}_1(x(t))\cdot \xi(t))^2}{\vert \nabla \tilde{u}_1(x(t)) \vert^2 \vert \xi(t) \vert ^2} \xi(t) -\frac{\nabla \tilde{u}_1(x(t)) \cdot \xi(t)}{\vert \nabla \tilde{u}_1(x(t)) \vert^2} \nabla \tilde{u}_1(x(t))\right).
\end{gather}
Note that $\frac{dx(t)}{dt} \cdot \xi(t)=0$, which means that $x(t)$ is a curve perpendicular to the direction in which ellipticity is lost.
\end{proof}
This theorem implies that singularities can only propagate in directions perpendicular to the direction in which ellipticity is lost.

\subsubsection{The normal operator}\label{sec:propnormaloperator}
For the analysis of propagation of singularities for the normal operator we restrict ourselves to the  situation with two measurements ($J=2$) in two dimensions ($n=2$). If one had a single measurement ($J=1$) the normal formulation would not need to be introduced and if more than two measurements were available ($J>2$), then one can ensure that the normal operator is elliptic by a specific choice of boundary conditions. This was also discussed briefly in Sec.~\ref{sec:classifi}. When $J=2$, the normal formulation is the equation
\begin{gather}\label{eq:lsj2}
(P_1^*P_1+P_2^*P_2) \delta \sigma = P_1^* (H_1-\tilde{H}_1) + P_2^* (H_2-\tilde{H}_2).
\end{gather}
The propagation of singularities depends on the symbol of $P_1^*P_1+P_2^*P_2$. The normal form increases the order of the zeros of the principal symbol, as seen by~\eqref{eq:prinsymls}. As a result, the principal symbol of $P_1^*P_1+P_2^*P_2$ is \emph{not} of principal type, and the classical propagation theory does not provide any information on the propagation of singularities. Instead we are going to use the theory for principal symbols of constant multiplicity following ~\cite{chazarain1976reflection}: 
\begin{definition}
The symbol $\mathfrak{q}(x,\xi)$ has constant multiplicity, $s \in \mathbb{N}$, at $(x_0,\xi_0)\in T^*\Omega\setminus 0$ if there exists a symbol $\acute{\mathfrak{q}}(x,\xi)$ of real principal type such that
\begin{gather}
\mathfrak{q}(x,\xi) = [\acute{\mathfrak{q}}(x,\xi)]^s,
\end{gather} 
in a neighbourhood of $(x_0,\xi_0)$
\end{definition}
In the classical theory on the propagation of singularities, operators of real principal type are characterised by the fact that the propagation of singularities is completely determined by the principal symbol. For operators with a principal symbol of constant multiplicity, a similar characterisation can be given if the operator satisfies the \emph{Levi condition}; that is a condition on the lower order terms. 
In this text we will not discuss the actual theory behind this condition, but just utilize its implications for operators of constant multiplicity. The following theorem of propagation of singularities for pseudo-differential operators with symbols having constant multiplicity and satisfying the Levi condition is due to Chazarain~\cite{chazarain1974propagation}:
\begin{thm}\label{thm:chazarain}
Let $P$ be a pseudo-differential operator of order $m$ defined in the smooth open set $X$, with principal symbol $\mathfrak{p}^m$ of constant multiplicity satisfying the Levi condition. For $u \in \mathcal{D}'(X)$, let $Pu =f$, then $\text{WF}(u)\setminus \text{WF}(f)$ is contained in the set of points $(x_0,\xi_0)$ for which $\mathfrak{p}^m(x_0,\xi_0)=(0)$ and invariant under the bicharacteristic strips relative to the symbol $\acute{\mathfrak{p}}$.
\end{thm}
We first show that the operator $P_1^*P_1+P_2^*P_2$ can be decomposed in the form of a polynomial of an operator of real principal type. Then we show that this implies that $P_1^*P_1+P_2^*P_2$ is of constant multiplicity and that it satisfies the Levi condition. We then calculate the bicharacteristic curves along which the singularities can propagate.
\begin{lem}
Let $p>1$, $\Omega \subset \mathbb{R}^2$ and let $P_1^*P_1+P_2^*P_2$ be the pseudo-differential operator from the scalar normal equation~\eqref{eq:lsj2}. Assume that ellipticity is lost in the direction $v(x) \in \mathbb{R}^2, |v(x)| = 1$, in a neighbourhood of a point $x_0\in\Omega$. Define by $v^T(x)$ a unit vector perpendicular to $v(x)$. Then in a neighbourhood of $(x_0,v(x_0))$, the pseudo-differential operator $P_1^*P_1+P_2^*P_2$ can be factorised, such that
\begin{gather}
P_1^*P_1+P_2^*P_2 = (EQ)^2+FQ+G,
\end{gather}
where $E\in \Psi^{-1}(\Omega)$ is elliptic, $Q\in \Psi^{1}(\Omega)$ is an operator with principal symbol $\xi \cdot v^T(x)$ and $F,G \in \Psi^{-2} (\Omega)$. 
\end{lem}

\begin{proof}
The full symbol of $P_1^*P_1+P_2^*P_2$ is given by
\begin{align}
|\mathfrak{p}_1|^2 + |\mathfrak{p}_2|^2 &= |\mathfrak{p}^0_1+ \mathfrak{p}^{-1}_1+\mathcal{O}(|\xi|^{-2})|^2+ |\mathfrak{p}^0_2+ \mathfrak{p}^{-1}_2+\mathcal{O}(|\xi|^{-2})|^2.
\end{align}
Since the principal symbols $\mathfrak{p}^0_j$ are real, we find that
\begin{align}
|\mathfrak{p}_1|^2 + |\mathfrak{p}_2|^2 &= (\mathfrak{p}^0_1)^2+(\mathfrak{p}^0_2)^2+2\left(\mathfrak{p}_1^0\text{Re}(\mathfrak{p}^{-1}_1)+\mathfrak{p}_2^0\text{Re}(\mathfrak{p}^{-1}_2)\right)+\mathcal{O}(|\xi|^{-2}).
\end{align}
By the definition of $v$ we know that
\begin{gather}
 (\nabla \tilde{u}_j \cdot v)^2 = \frac{1}{p} |\nabla \tilde{u}_j|^2\quad \text{and}\quad 
  (\nabla \tilde{u}_j \cdot v^T)^2 = \left(1 - \frac{1}{p}\right) |\nabla \tilde{u}_j|^2.
\end{gather}
Since $\{v,v^T\} = \{v(x),v(x)^T\}$ is an orthonormal basis for $\mathbb{R}^2$ we have
\begin{align}
|\xi|^2 &= (\xi \cdot v)^2 + (\xi \cdot v^T)^2,\\
(\nabla \tilde{u}_j \cdot \xi)^2& = (\nabla \tilde{u}_j \cdot v)^2(v \cdot \xi)^2+(\nabla \tilde{u}_j \cdot v^T)^2(v^T \cdot \xi)^2\\&+2(\nabla \tilde{u}_j \cdot v)(v \cdot \xi)(\nabla \tilde{u}_j \cdot v^T)(v^T \cdot \xi).
\end{align}
This makes it possible to write the principal symbol of $P_j$ as
\begin{align}
\mathfrak{p}^0_j &= \vert \nabla \tilde{u}_j \vert^p \left(  1 - p \frac{ (\nabla \tilde{u}_j \cdot \xi)^2}{\vert \nabla \tilde{u}_j \vert^{2}|\xi| ^2} \right)\\
 &= \frac{\vert \nabla \tilde{u}_j \vert^p}{|\xi| ^2}\left( (2-p)(v^T \cdot \xi) ^2 - \frac{2p}{|\nabla \tilde{u}_j |^2}(\nabla \tilde{u}_j \cdot v)(v \cdot \xi)(\nabla \tilde{u}_j \cdot v^T)(v^T \cdot \xi)\right)\\
  &=  \mathfrak{e}_j(v^T \cdot \xi).
\end{align}
Here $\mathfrak{e}_j \in S^{-1}(\Omega)$ is non-zero in a neighbourhood of $(x_0,v(x_0))$, because for $p>1$ neither $(\nabla \tilde{u}_j \cdot v)$ nor $(\nabla \tilde{u}_j \cdot v^T)$ vanishes in such a neighbourhood. The principal symbol of $P_1^*P_1+P_2^*P_2$ then takes the simple form
\begin{align}
(\mathfrak{p}^0_1)^2+(\mathfrak{p}^0_2)^2 &= (\mathfrak{e}_1^2 + \mathfrak{e}_2^2)(v^T \cdot \xi)^2\\
&= \mathfrak{e}^2 (v^T \cdot \xi)^2,
\end{align}
where $\mathfrak{e} \in S^{-1}(\Omega)$ is also non-zero in a neighbourhood of $(x_0,v(x_0))$. In a similar way we can write
\begin{align}
2\left(\mathfrak{p}_1^0\text{Re}(\mathfrak{p}^{-1}_1)+\mathfrak{p}_2^0\text{Re}(\mathfrak{p}^{-1}_2)\right) &= 2\left(\mathfrak{e}_1\text{Re}(\mathfrak{p}^{-1}_1)+\mathfrak{e}_2\text{Re}(\mathfrak{p}^{-1}_2)\right)(v^T \cdot \xi)\\ &= \mathfrak{f}(v^T \cdot \xi),
\end{align}
where $\mathfrak{f} \in S^{-2}(\Omega)$. This shows that
\begin{gather}
|\mathfrak{p}_1|^2+|\mathfrak{p}_2|^2=\mathfrak{e}^2 (v^T \cdot \xi)^2 + \mathfrak{f}(v^T \cdot \xi) +\mathfrak{g} ,
\end{gather}
where $\mathfrak{g} \in S^{-2}(\Omega)$. It follows that
$P_1^*P_1+P_2^*P_2$ in a neighbourhood of $(x_0,v(x_0))$ can be written in the form
\begin{gather}
P_1^*P_1+P_2^*P_2 = (EQ)^2+FQ+G,
\end{gather}
where $E\in \Psi^{-1}(\Omega)$ is elliptic, $Q\in \Psi^{1}(\Omega)$ is an operator with principal symbol $\xi \cdot v^T$ and $F,G \in \Psi^{-2} (\Omega)$.
\end{proof}
The previous lemma implies that the operator $P_1^*P_1+P_2^*P_2$ is an operator of constant multiplicity because its principal symbol can be written in the form $\left(\mathfrak{e}\left(\xi\cdot v^T\right)\right)^2$, with $\mathfrak{e} \in S^{-1}(\Omega)$. This type of factorisation of $P_1^*P_1+P_2^*P_2$ also implies that the Levi condition is satisfied~\cite[Thm. 2.1]{chazarain1974propagation}. In turns out that the bicharacteristic curves are similar to those of both $P_1$ and $P_2$, as the following theorem shows.
\begin{thm}
For $p>1$, the bicharacteristic curves of $EQ$ are perpendicular to the directions in which ellipticity is lost.
\end{thm}
\begin{proof}
The principal symbol of $EQ$ is $\mathfrak{e}(\xi\cdot v^T)$, which is clearly of real principal type. We find that the curve $x(t)$ with 
\begin{gather}
 \left.\frac{dx(t)}{dt}\right\vert_{\mathfrak{e} (\xi\cdot v^T)=0} = \left.\frac{\partial}{\partial \xi} \mathfrak{e} (\xi\cdot v^T)\right\vert_{\mathfrak{e} (\xi\cdot v^T)=0}  = \mathfrak{e} v^T
\end{gather}
clearly is perpendicular to $v$ which defines the direction in which ellipticity is lost.
\end{proof}
As a result of Thm.~\ref{thm:chazarain} we can conclude that if $P_1^*P_1+P_2^*P_2$ is not elliptic and $p>1$, a solution $\delta \sigma$ to the equation~\eqref{eq:lsj2} has singularities propagating in directions perpendicular to the direction in which ellipticity is lost.

To further investigate how singularities propagate, we need to know the wave front set of the right-hand side of the normal equation for the linearised inverse problem. The following theorem shows how the wave front sets of $P_1^* (H_1-\tilde{H}_1) + P_2^* (H_2-\tilde{H}_2)$ and $\sigma$ are related.
\begin{thm}\label{thm:wfh}
If $\{\vert \nabla u_j \vert\}_{j=1}^2$ are all bounded below by a positive constant in $\Omega$ and $\tilde{\sigma}$ is smooth, then
\begin{gather}
\operatorname{WF}(P_1^* (H_1-\tilde{H}_1) + P_2^* (H_2-\tilde{H}_2)) \subseteq \operatorname{WF}(\sigma).
\end{gather}
\end{thm}
\begin{proof}
The wave front set of a sum is at most the union of their respective wave front sets~\cite[Thm.~IX.44]{reed1975methods}. This implies
\begin{gather}
\operatorname{WF}(P_1^* (H_1-\tilde{H}_1) + P_2^* (H_2-\tilde{H}_2)) \subseteq
\operatorname{WF}(P_1^* (H_1-\tilde{H}_1)) \cup \operatorname{WF}( P_2^* (H_2-\tilde{H}_2)).
\end{gather}
Because pseudo-differential operators can at most decrease the wave front set, we have the inclusions
\begin{gather}\label{eq:proofwf1}
\operatorname{WF}(P_j^* (H_j-\tilde{H}_j))\subseteq\operatorname{WF}(H_j), \quad j\in\{1,2\}.
\end{gather}
Now, for the operator $L_\sigma := \nabla \cdot (\sigma \nabla [\cdot])$ we have that
\begin{gather}
\operatorname{Char}(L_\sigma) \subseteq \operatorname{WF}(\sigma),
\end{gather}
which implies
\begin{gather}\label{eq:proofwf2}
\operatorname{WF}(u_j) \subseteq \operatorname{WF}(\sigma),\quad j\in\{1,2\}.
\end{gather}
Consider the operator $P$ defined by
\begin{gather}
Pv := \sigma|\nabla v|^p,\quad p>0.
\end{gather}
For $|\nabla v|$ uniformly boundary below by a positive constant, $P$ is elliptic away from $\operatorname{WF}(\sigma) \cup \operatorname{WF}(v)$, and it follows by the pseudo-local nature of pseudo-differential operators that
\begin{gather}\label{eq:proofwf3}
\operatorname{WF}(Pv) \subseteq \operatorname{WF}(\sigma) \cup \operatorname{WF}(v).
\end{gather}
Note that $H_j = P u_j = \sigma | \nabla u_j |^p$. Therefore, from~\eqref{eq:proofwf2} and~\eqref{eq:proofwf3} we get the relation
\begin{gather}
\operatorname{WF}(H_j) \subseteq \operatorname{WF}(\sigma),\quad j\in\{1,2\},
\end{gather}
which by~\eqref{eq:proofwf1} completes the proof.
\end{proof}
\section{Numerical implementation of the linearised inverse problem}
In this chapter we present an implementation of the linearised inverse problem which we use to visualise the results on ellipticity and propagation of singularities. It should be noted that the implemented system is slightly different from, but equivalent to, the normal formulation analysed in the previous chapter. The difference lies in the fact that we do not consider the scalar formulation of the inverse problem, because a numerical implementation of a non-local pseudo differential operator is difficult. Instead we consider the full differential system which can be implemented using the finite element method.
\subsection{The forward problem}
The implementation of the forward problem is used to generate the data sets $\{H_j-\tilde{H}_j\}_{j=1}^J$ and the gradients of the reference potentials $\{\tilde{u}_j\}_{j=1}^J$. Let $L_+^\infty(\Omega)$ denote the space of function bounded above and below by positive constants in $\Omega$. For $\sigma \in L_+^\infty(\Omega)$, we want to solve for $\nabla u_j$ in the linear inhomogeneous problem \eqref{eq:condprob}
with $f_j \in C^\infty(\partial \Omega)$. Let $F_j \in C^\infty(\bar{\Omega})$ be some function satisfying $F_j=f_j$ on $\partial \Omega$. Because $\Omega$ is a bounded open domain with smooth boundary we know that such a function exists. 
By introducing the function $v_j=u_j-F_j$ and the vector function $\boldsymbol w_j$, we can reformulate~\eqref{eq:condprob} as the homogeneous mixed problem
\begin{gather}\label{eq:mixedforward}
\left\{\begin{aligned}
\nabla \cdot  \boldsymbol w_j &= -\nabla \cdot (\sigma \nabla F_j) &&\text{in } \Omega,\\
\boldsymbol w_j &= \sigma \nabla v_j &&\text{in } \Omega,\\
v_j &= 0 &&\text{on } \partial \Omega.
\end{aligned}
\right.
\end{gather}
The solution $\{v_j, \boldsymbol w_j\}$ to~\eqref{eq:mixedforward} then defines the pair $\{u_j, \nabla u_j\}$ using the definition of $v_j$. 

Let $(\cdot,\cdot)$ denote the $L^2(\Omega)$ inner product. We then have the mixed (weak) formulation of~\eqref{eq:mixedforward}:
Find $\{v_j, \boldsymbol w_j\} \in  \{H_0^1(\Omega),H_\text{div}(\Omega)\} $ such that
\begin{gather*}
\left(\sigma^{-1}\boldsymbol  w_j,\boldsymbol \psi\right) + 
(v_j, \nabla \cdot \boldsymbol \psi)+
(\nabla \cdot \boldsymbol w_j, \phi)  = \left(\sigma \nabla F_j, \nabla \phi\right),\\\forall \,\{\phi,\boldsymbol \psi\} \in \{H^1_0(\Omega), H_\text{div}(\Omega)\}.
\end{gather*}
Here $H_{\text{div}}(\Omega)$ denotes the usual (Hilbert) space of vector functions in $[L^2(\Omega)]^n$ for which the divergence is also an $L^2(\Omega)$ function (see e.g.~\cite[Chap. 20]{tartar2007introduction}).

We solve this problem using the finite element method. In the finite dimensional setting we use continuous Galerkin elements which are conforming in $H^1$ and Raviart--Thomas elements which are conforming in $H_\text{div}$~\cite{logg2012automated}. Once the discrete approximate of the pair $\{v_j, \boldsymbol w_j\}$ has been found we can determine an approximation of $\{u_j,\nabla u_j\}$ using the definition of $v_j$. For this we need $\nabla F_j$ which should be implemented using its analytical expression (if it is known) to avoid any unnecessary errors from numerical differentiation. In our case we choose $F_j$ as the harmonic extension of $f_j$, for which the analytic expression is known. 




\subsection{The linearised inverse problem}
To derive an appropriate weak formulation, we use the least squares finite element method to find the $L^2$-minimizer of the original equation. In other words, we define the least squares solution to the collection of second order problems~\eqref{eq:linsystem}, as a function $\boldsymbol x$ that minimizes
\begin{gather}
 \Vert \mathcal{A} \boldsymbol x - \boldsymbol b\Vert_{L^2(\Omega)}^2,
\end{gather}
where
\begin{gather}\label{eq:amatrix}
\mathcal{A} = \begin{bmatrix}
\nabla \cdot ([\cdot] \nabla \tilde{u}_1) & \nabla \cdot (\tilde{\sigma}  \nabla [\cdot]) & \cdots & 0 \\
\vdots & \vdots & \ddots & \vdots\\
\nabla \cdot ([\cdot] \nabla \tilde{u}_J) & 0 & \cdots &  \nabla \cdot (\tilde{\sigma}  \nabla [\cdot])\\
\vert \nabla \tilde{u}_1 \vert^p  &  p\tilde{\sigma} \frac{\nabla \tilde{u}_1  \cdot \nabla [\cdot]  }{\vert \nabla \tilde{u}_1 \vert^{2-p}}& \cdots & 0\\
\vdots & \vdots  & \ddots & \vdots\\
\vert \nabla \tilde{u}_J \vert^p  & 0  & \cdots & p\tilde{\sigma} \frac{\nabla \tilde{u}_J \cdot \nabla [\cdot]  }{\vert \nabla \tilde{u}_J \vert^{2-p}}
\end{bmatrix},
\boldsymbol x=\begin{bmatrix}
\delta \sigma\\
\delta u_1\\
\vdots\\
\delta u_J
\end{bmatrix}
\\
\text{and} \quad
\boldsymbol b = \begin{bmatrix}
0\\
\vdots\\
0\\
H_1 - \tilde{H}_1\\
\vdots\\
H_J -\tilde{H}_J
\end{bmatrix},
\end{gather}
for $\boldsymbol x$ satisfying homogeneous boundary conditions for $\{\delta u_j\}_{j=1}^J$ in appropriate discrete function spaces. A numerical implementation of this system is not practical, because it turns out that $H^2$-conforming finite element spaces are necessary to approximate the functions $\{\delta u_j\}_{j=1}^J$. This follows from the fact that $\mathcal{A}$ is a second order operator. $H^2$-spaces are known to be impractical to implement and as a technical step~\eqref{eq:linsystem} is therefore expressed in the mixed form of a first order system
\begin{gather}\label{eq:mixedformula}
\left\{\begin{aligned}
\hat{\mathcal{A}} \hat{\boldsymbol x}&= \boldsymbol b \\
\mathcal{M} \boldsymbol x  &= \hat{\boldsymbol x}
\end{aligned}
\begin{aligned}
&\text{ in } \Omega,\\
&\text{ in } \Omega,
\end{aligned}
\right.
\end{gather}
where $\mathcal{M} = \operatorname{diag}(1, \nabla, \ldots, \nabla)$ and $\hat{\mathcal{A}}$ is defined by the relation $\mathcal{A} =\hat{\mathcal{A}} \mathcal{M}$. A simple analysis of the operators $\hat{\mathcal{A}}$ and $ \mathcal{M}$ shows that they act as the operators $\hat{\mathcal{A}}: H^1(\Omega) \times [H_\text{div}(\Omega)]^J \rightarrow [L^2(\Omega)]^{2J}$, $\mathcal{M}: H^1(\Omega) \times [H^1(\Omega)]^J \rightarrow H^1(\Omega) \times [L^2(\Omega)]^J$ respectively. Thus, they specify the appropriate function space setting which satisfies the boundary conditions:
\begin{gather}
\{\boldsymbol x , \boldsymbol{\hat{x}}\} \in  \left\{ H^1(\Omega) \times [H_0^1(\Omega)]^J ,H^1(\Omega) \times [H_\text{div}(\Omega)]^J  \right\}.
\end{gather}
We suggest to solve \eqref{eq:mixedformula} in the least squares sense by minimizing the energy functional
\begin{gather}
I( \boldsymbol \tau, \hat{\boldsymbol \tau} ) = \Vert \hat{\mathcal{A}} \hat{\boldsymbol \tau} - \boldsymbol b\Vert_{L^2(\Omega)}^2+\Vert \mathcal{M}\boldsymbol \tau -  \hat{\boldsymbol \tau}\Vert_{L^2(\Omega)}^2.
\end{gather}
This leads to the least squares weak formulation of~\eqref{eq:mixedformula}: Find\\$\{\boldsymbol x, \hat{\boldsymbol x}\} \in \left\{ H^1(\Omega) \times [H_0^1(\Omega)]^J ,H^1(\Omega) \times [H_\text{div}(\Omega)]^J  \right\}$ such that
\begin{gather}\label{eq:systemsolve}
(\hat{\mathcal{A}} \hat{\boldsymbol x} ,\hat{\mathcal{A}} \hat{\boldsymbol \phi}) + (\mathcal{M} \boldsymbol x- \boldsymbol{\hat{x}} , \mathcal{M} \boldsymbol \phi- \hat{\boldsymbol \phi}) = (\boldsymbol b,\hat{\mathcal{A}}\hat{\boldsymbol \phi}),\\ \forall \, \{\boldsymbol \phi, \hat{\boldsymbol \phi}\} \in \left\{ H^1(\Omega) \times [H_0^1(\Omega)]^J ,H^1(\Omega) \times [H_\text{div}(\Omega)]^J  \right\}.
\end{gather}
Note that the essential homogeneous Dirichlet boundary condition for $\{\delta u_j\}_{j=1}^J$ are imposed by the chosen function spaces $H_0^1(\Omega)$. Additionally, the weak formulation implicitly requires that the so-called concomitant of $\hat{\mathcal{A}}\boldsymbol{\hat{x}}-\boldsymbol{b}$ and $\boldsymbol{\hat{\phi}}$ vanishes. This is a direct consequence of the least squares finite element method formulation~\cite[Rem. 3.12]{bochev2009least-squares} and therefore the problem~\eqref{eq:systemsolve} can be identified as a weak formulation of the first order equivalent of the normal equation
\begin{gather}
\mathcal{A}^*\mathcal{A} \boldsymbol{x} =\mathcal{A}^* \boldsymbol{b} \,  \text{ in } \Omega,
\end{gather}
for $\boldsymbol{x} =\{x_1, \ldots, x_{J+1}\} \in  H^1(\Omega) \times [H_0^2(\Omega)]^J$ satisfying the additional boundary condition
\begin{gather}
\{\nabla \cdot (x_1 \nabla \tilde{u}_1) + \nabla \cdot (\tilde{\sigma}  \nabla x_j)\}_{j=2}^{J+1} = 0\,  \text{ on } \partial \Omega.
\end{gather}
Thus, a solution $\delta \sigma$ to the equation~\eqref{eq:lsform} for which we have shown when and how singularities propagate correponds to the top element of a solution vector to the problem~\eqref{eq:systemsolve}. Again, in the finite dimensional setting we use continuous Galerkin elements which are conforming in $H^{1}$ and Raviart--Thomas elements which are conforming in $H_\text{div}$. 


\subsection{Software, geometry, mesh, visualization, phantom}
The implementation of the discrete finite dimensional weak formulations of the forward problem and the linearised inverse problem has been done using The FEniCS Project, a collection of free software with an extensive list of features for automated, efficient solving of partial differential equations~\cite{logg2012automated}.

We will consider the problem in a 2D unit disc geometry, because it is relatively simple to implement and it satisfies the smoothness requirements of the boundary. Meshing is done automatically in FEniCS and the number of elements are dependent on a chosen mesh parameter. For this work we have used a mesh of 48\,400 elements because it provides a good compromise between the finite element methods ability to capture local phenomena and the computational cost. It should, however, be noted that the forward problem is solved on a different mesh with 49\,284 elements. The calculated interior data functionals and the reference potentials are afterwards projected onto the mesh of 48\,400 elements. This is done to avoid committing an inverse crime. For all visualizing purposes we use the open-source visualization software Paraview~\cite{henderson2004the}. 

We choose to linearise around $\tilde{\sigma} = 1$ because the theoretical results presented in Chapter 3 relies on the fact that $\tilde{\sigma}$ is smooth. Additionally it simplifies the expressions for $\{\nabla \tilde{u}_j\}_{j=1}^J$ which makes it very easy to predict the directions in which ellipticity is lost and thereby in which directions singularities can propagate. As a phantom we use the function $\sigma=1$ which is perturbed by $0.1$ in a small rectangular domain, see Fig.~\ref{fig:phantom}. 

\subsection{What to expect}
By the results of Thms.~\ref{thm:pjell} and~\ref{thm:pj2ell} the ellipticity of the linearised inverse problem depends on the direction of the vectors $\{\nabla \tilde{u}_j \}_{j=1}^J$. If the boundary condition $f_j$ is chosen as a linear combination of the coordinate functions, $\nabla \tilde{u}_j$ will be constant in the interior of $\Omega$ because $\tilde{\sigma}$ is constant. Furthermore, $\vert \nabla \tilde{u}_j \vert$ will be bounded below by a positive constant for any $j$ as required by some of the theorems. Thus, if $p\geq 1$, it is not a difficult task to find two boundary conditions such that the system loses ellipticity in known directions.

For the phantom presented in the previous subsection it is a standard exercise to show that the wave front set is limited to the discontinuous boundary of the perturbation and the corresponding $\xi$ vectors are normal to the boundary. In the corners of the perturbation the function is singular in the directions $\xi$ pointing inside the perturbation and the opposite directions pointing outwards. Some of the wave front set is depicted in Fig.~\ref{fig:wavefrontset}, where the starting point of each arrow is the location of the singularity in $\Omega$. To keep the figure simple, only the outwards pointing directions are depicted. 

If $p>1$, the loss of ellipticity will manifest itself as visible artifacts oriented perpendicular to the direction in which ellipticity is lost. 
As an example, if ellipticity is lost in the horizontal direction, singularities will propagate in the vertical direction. It is expected that the reconstruction will be a somewhat smoother representation, due to the additional regularity imposed by the discrete least squares formulation.
\begin{figure}[!htb]
        \centering
        \begin{subfigure}[b]{0.42\textwidth}
                \centering
                \includegraphics[width=\textwidth]{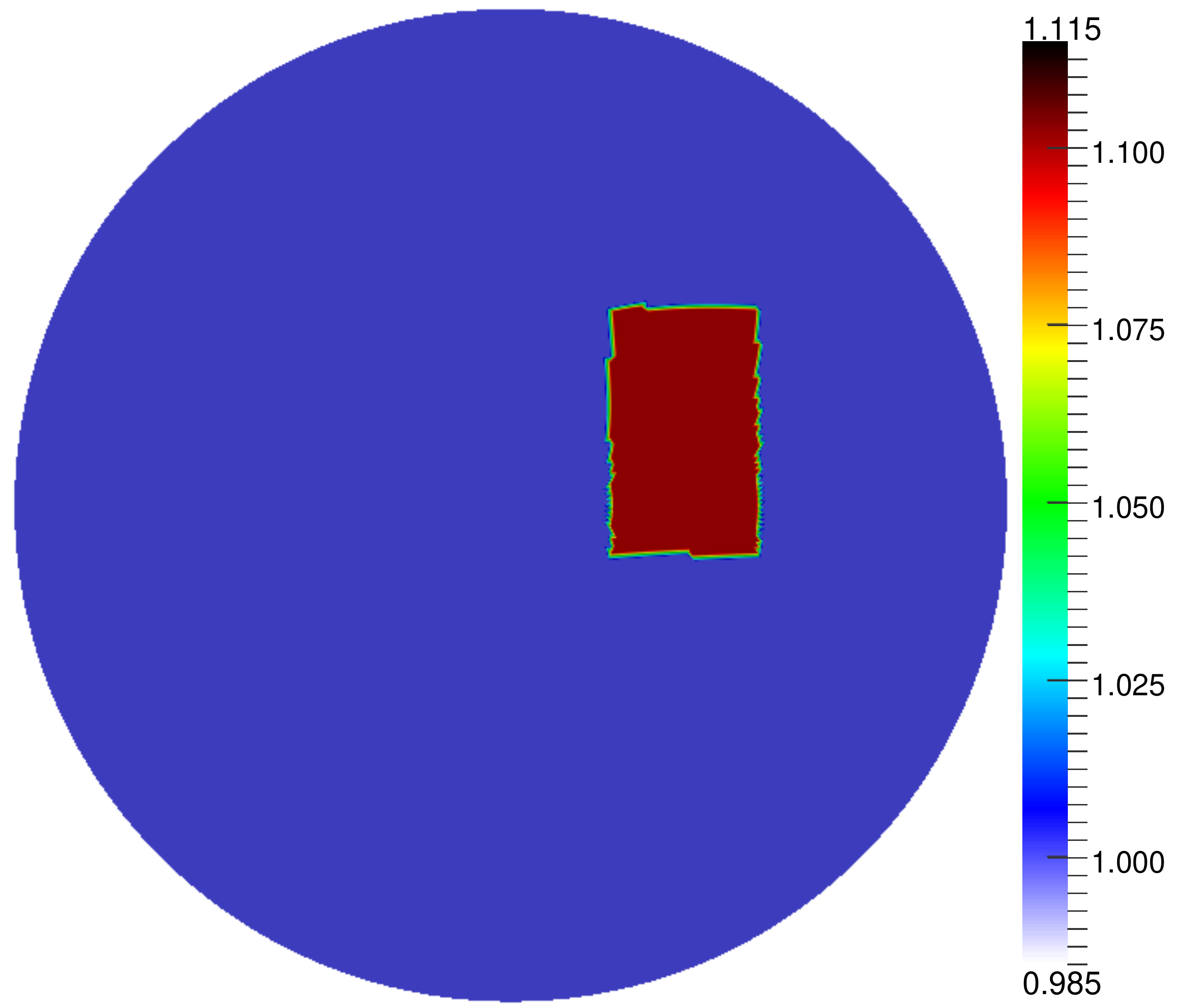}
                \caption{The phantom $\sigma$. Notice the discontinuity along the boundary of the perturbation.}
                \label{fig:phantom}
        \end{subfigure}
        \,\,
           \begin{subfigure}[b]{0.46\textwidth}
                \centering
                \includegraphics[width=0.75\textwidth]{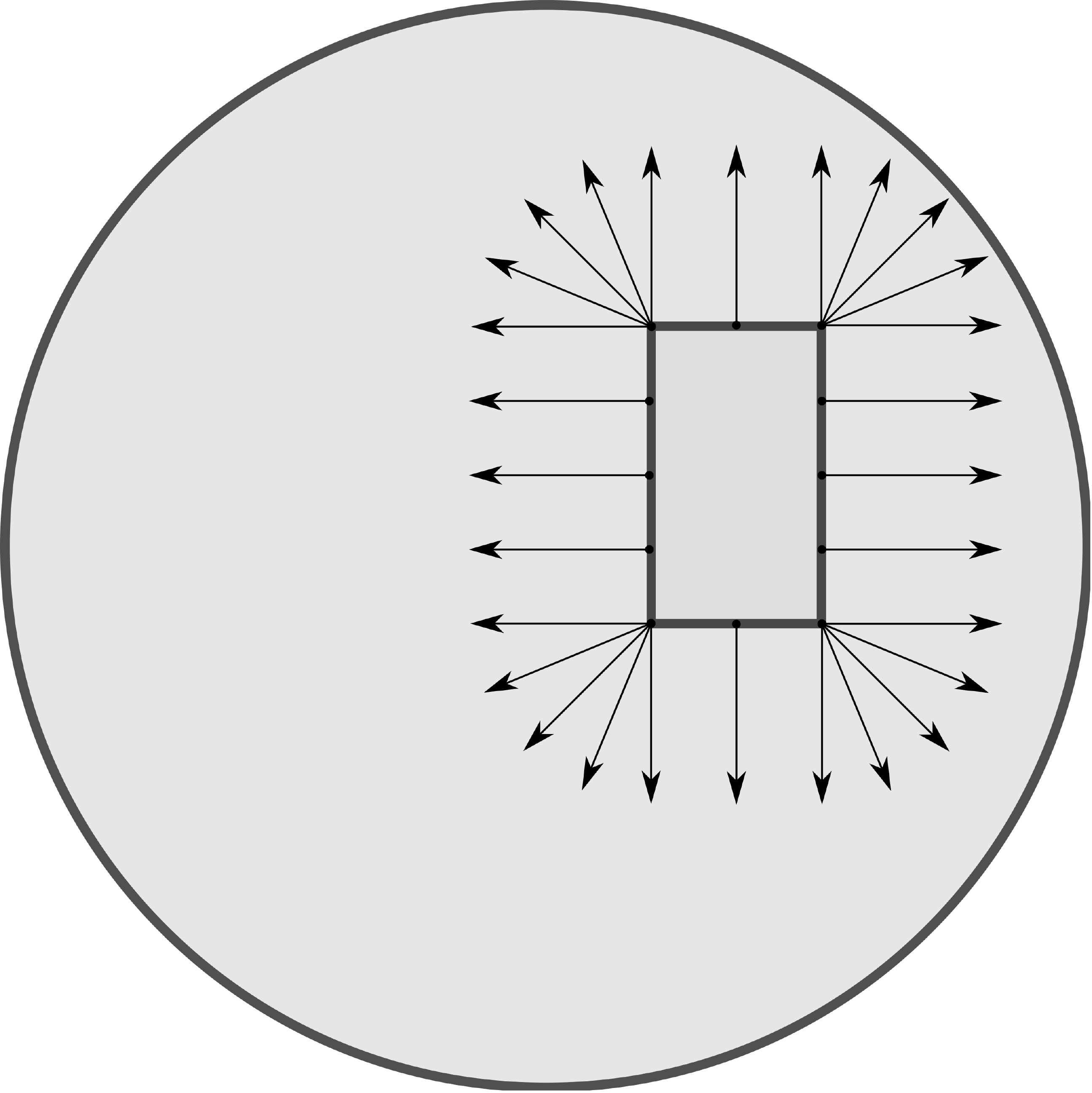}
                \caption{The corresponding wave front set of the phantom. Only some of the outwards pointing directions are depicted.}
                \label{fig:wavefrontset}
        \end{subfigure}
        \caption{The phantom used in the numerical implementation along with an illustration of some of the corresponding wave front set.}\label{fig:phwavefr}
\end{figure}

\section{Numerical results}
In this chapter we present some results of the numerical implementation. We examine how the solution changes when the operator loses ellipticity and also see how the propagation of singularities depends on the modelling parameter~$p$. We also analyse how a clever choice of boundary conditions can reduce the propagation of singularities when the problem is not elliptic.
\subsection{The elliptic and non-elliptic case}
Fig. 2 consists of plots of $\delta \sigma$ for two values of the parameter $p$. For each value we have plotted the reconstruction when the problem is elliptic and when the problem is not elliptic. 
In the elliptic case the reconstructions are close to being perfect, and non-zero values in $\delta \sigma$ is limited to the location of the perturbation. This is very different from the non-elliptic case, where we see that $\delta \sigma$ indeed has several visible artifacts.
The presented results on the propagation of singularities only holds for $p>1$. When $p=1$ we have the borderline case, between a perfect reconstruction modulo smoothing terms (the elliptic case, $p<1$) and the situation where singularities propagate (the hyperbolic case, $p>1$). With a numerical implementation it seems reasonable to expect some kind of combination of these two scenarios in the reconstruction. 
Indeed, the reconstruction presented in Fig.~\ref{fig:mousefw} also shows what looks like a smooth representation of propagation in the vertical direction.
\begin{figure}[!htb]
        \centering
        \begin{subfigure}[b]{0.22\textwidth}
                \centering
                \includegraphics[width=\textwidth]{ep1f2-eps-converted-to.pdf}
                \caption{Elliptic\\($p=1, J=2$)}
                \label{fig:mosdfuse}
        \end{subfigure}
           \begin{subfigure}[b]{0.22\textwidth}
                \centering
                        \captionsetup{justification=centering}
                \includegraphics[width=\textwidth]{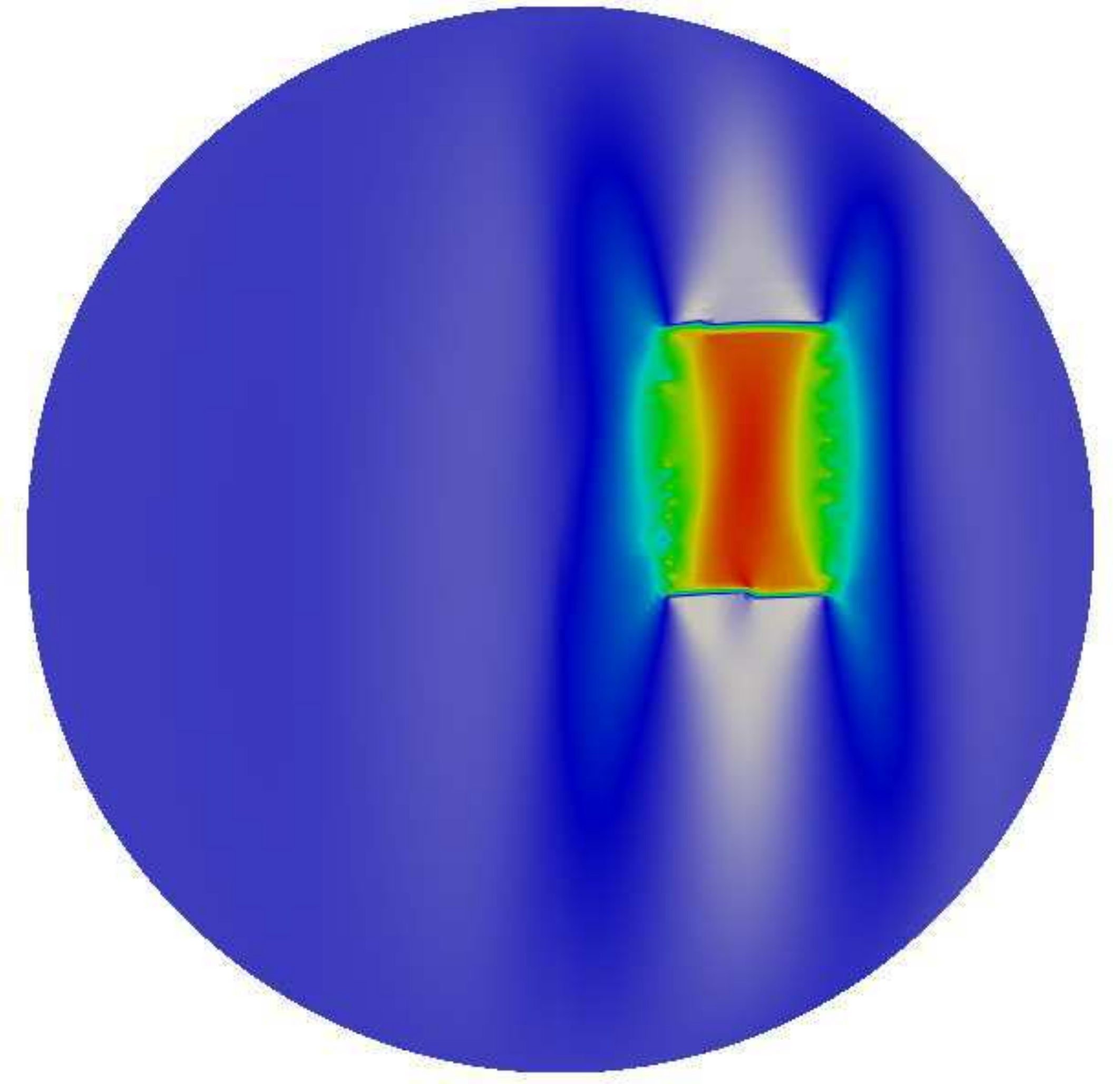}
                \caption{Non-elliptic\\($p=1, J=1$)}
                \label{fig:mousefw}
        \end{subfigure}
        \begin{subfigure}[b]{0.22\textwidth}
                \centering
                \includegraphics[width=\textwidth]{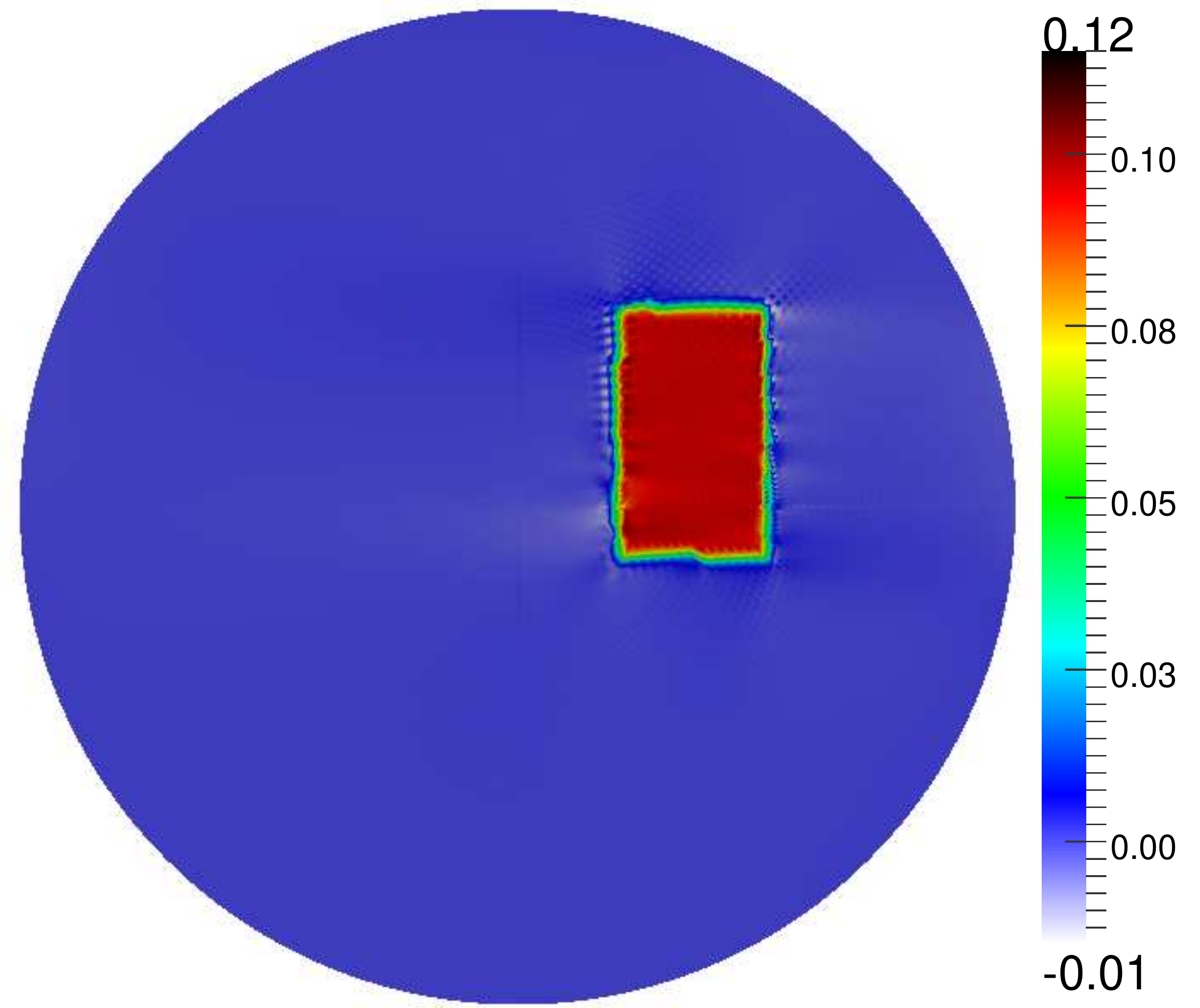}
                \caption{Elliptic\\($p=2, J=2$)}
                \label{fig:mou24fse}
        \end{subfigure}
           \begin{subfigure}[b]{0.215\textwidth}
                \centering
                \includegraphics[width=\textwidth]{nep2f1-eps-converted-to.pdf} \captionsetup{justification=centering}
                \caption{Non-elliptic\\($p=2, J=1$)}
                \label{fig:mousffdse}
        \end{subfigure}
                   \begin{subfigure}[b]{0.038\textwidth}
                \centering
                \includegraphics[width=\textwidth]{colorbar-eps-converted-to.pdf}
             \caption*{}
        \end{subfigure}
        \caption{Reconstructions of $\delta \sigma$ in the elliptic and non-elliptic case for two different values of $p$. The reconstructions clearly show that an non-elliptic problem results in a reconstruction with several artifacts. The difference between the reconstructions when the problem is elliptic and non-elliptic is very clear from the plots.}\label{fig:ellnonell}
\end{figure}

\subsection{The dependence on $p$}
The directions in which ellipticity is lost depend on the parameter $p$ and the directions of $\{\nabla \tilde{u}_j \}_{j=1}^J$. For $J=1, p>1$ there will always be two directions along which ellipticity is lost. These directions will be at angles $\pm \cos^{-1}(p^{-\frac{1}{2}})$ relative to $\nabla \tilde{u}_1$. To verify this dependence on the parameter~$p$ we solve the system for different values of $p$ with $\nabla \tilde{u}_1$ fixed in the horizontal direction, see Figure~\ref{fig:animalsss}. On the top of the figure we have placed arrows which show the predicted directions in which singularities propagate. We clearly see that the propagation of singularities follow the theoretical results. 
\begin{figure}[!htb]
        \centering
        \begin{subfigure}[b]{0.22\textwidth}
                \centering
                \includegraphics[width=\textwidth]{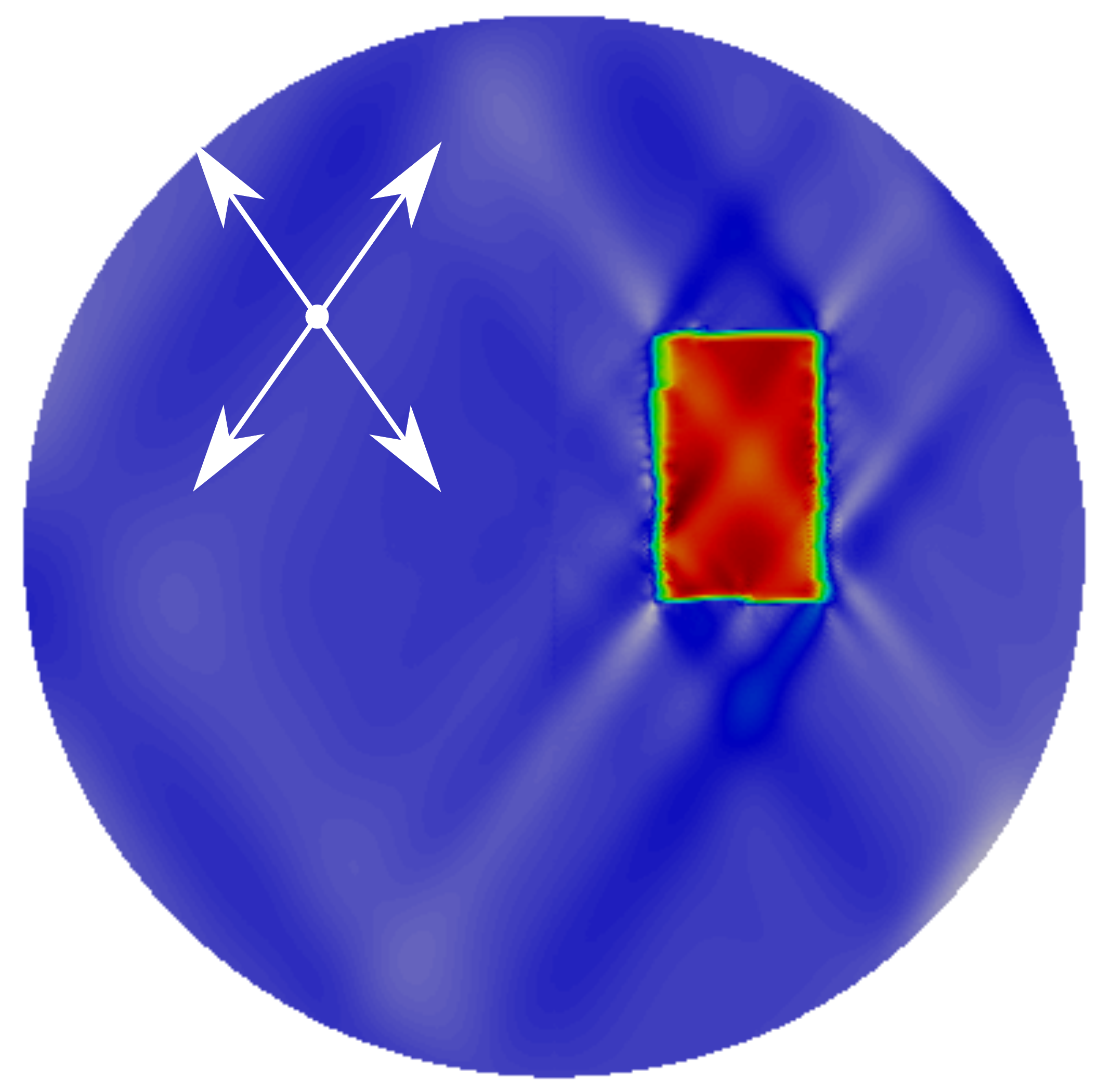}
                \caption{$p=\frac{3}{2}$}
                \label{fig:asdf}
        \end{subfigure}
              \begin{subfigure}[b]{0.22\textwidth}
                \centering
                \includegraphics[width=\textwidth]{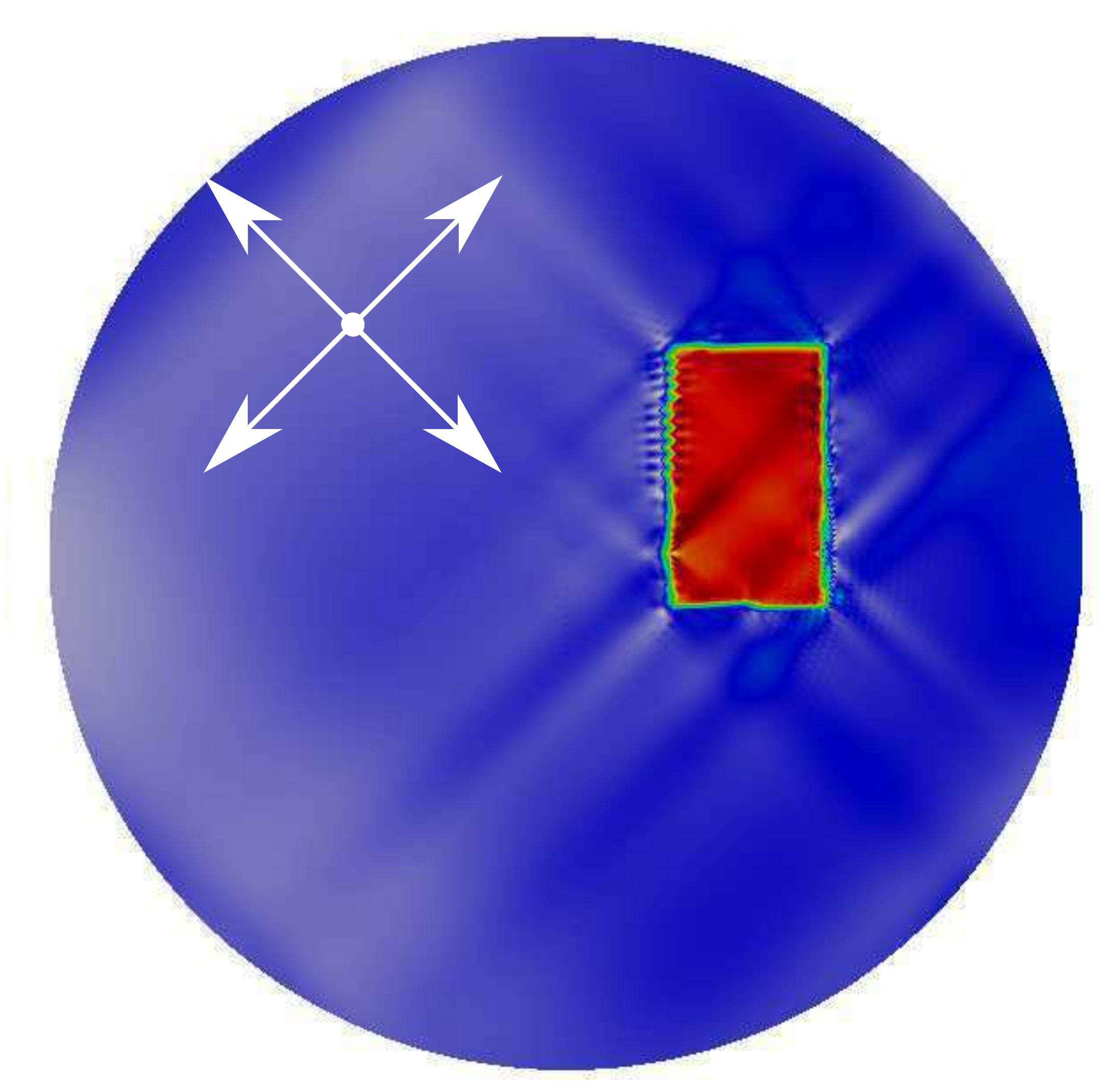}
                \caption{$p=2$}
                \label{fig:asdfgg}
        \end{subfigure}
        \begin{subfigure}[b]{0.22\textwidth}
                \centering
                \includegraphics[width=\textwidth]{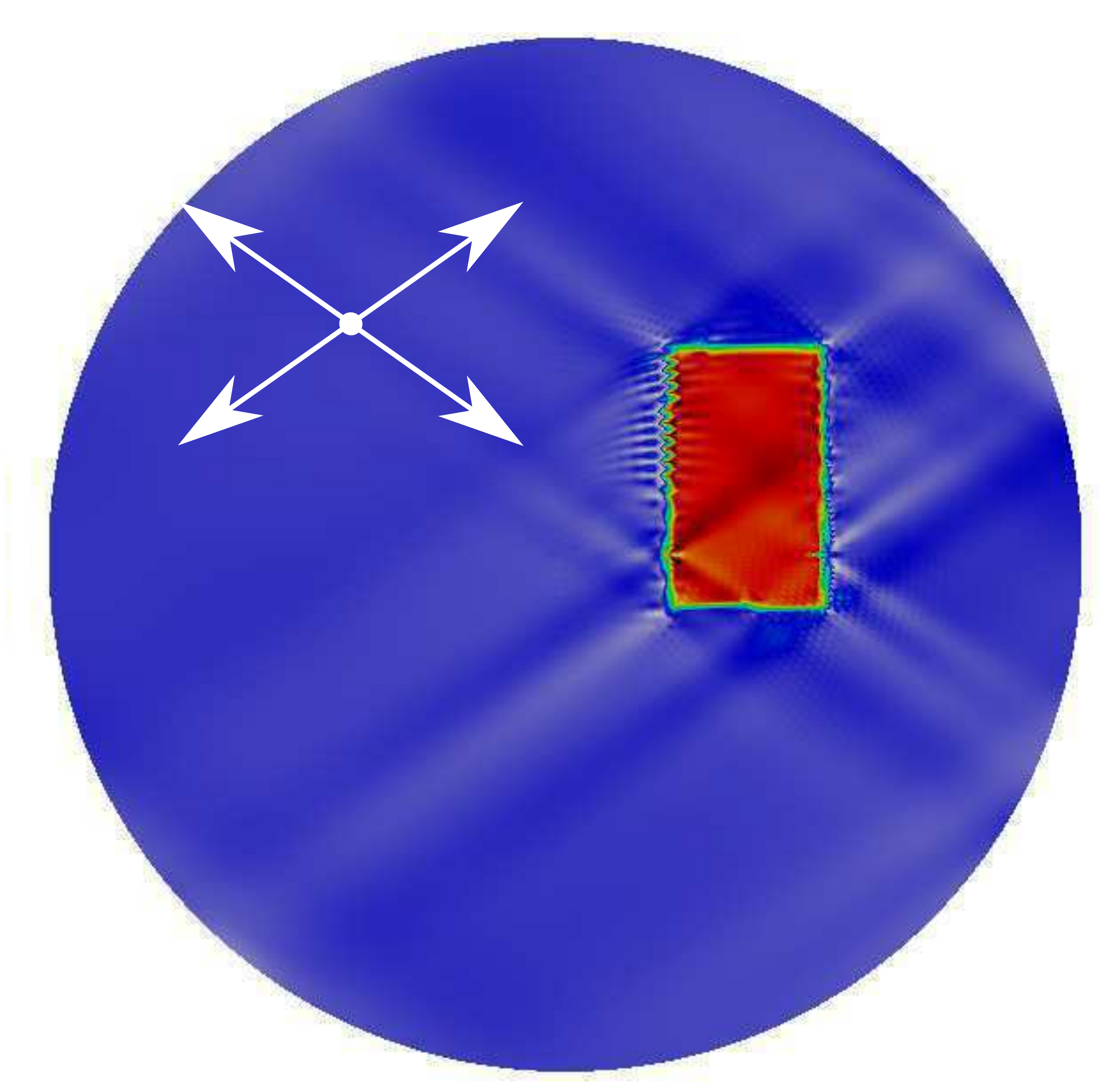}
                \caption{$p=3$}
                \label{fig:mouse55}
        \end{subfigure}
           \begin{subfigure}[b]{0.22\textwidth}
                \centering
                \includegraphics[width=\textwidth]{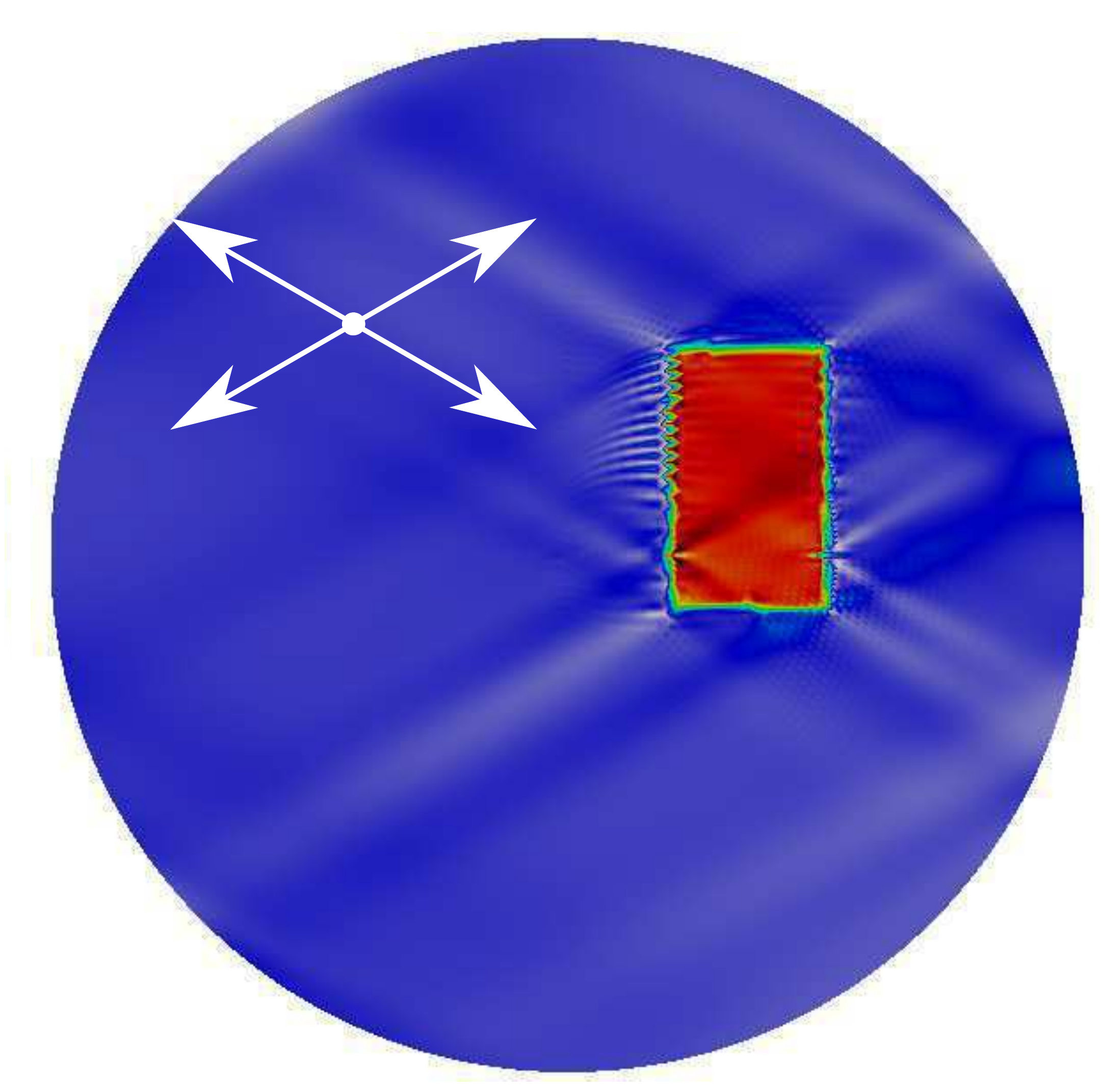}
                \caption{$p=4$}
                \label{fig:mouse32}
        \end{subfigure}
         \begin{subfigure}[b]{0.038\textwidth}
                \centering
                \includegraphics[width=\textwidth]{colorbar-eps-converted-to.pdf}
      
             \caption*{}
        \end{subfigure}
       \caption{Propagation of singularities for a single measurement, where $\nabla \tilde{u}_1$ points in the horizontal direction. The directions in which ellipticity is lost, and therefore also the directions of propagation of singularities, depend on the parameter $p$. }\label{fig:animalsss}
\end{figure}
\subsection{The interaction with the wave front set}
Compared to Fig.~\ref{fig:ellnonell}, the reconstructions in Fig.~\ref{fig:animalsss} seem less affected by the loss of ellipticity. This is because only the singularities located in the corners of the rectangular perturbation have a direction that aligns with the direction in which singularities propagate. Thus, the quality of the reconstruction depends on the direction in which ellipticity is lost and the presence of those directions in the wave front set of the data. As seen in Fig.~\ref{fig:wavefrontset}, the phantom has singularities mainly in the horizontal and vertical direction. If such knowledge about the singularities of the phantom is available, one can choose boundary conditions that could limit the possible effect of propagation of singularities. As an example, when the propagation of singularities is neither in the vertical nor in the horizontal direction, singularities cannot propagate along the boundary of the rectangular perturbation and then one would expect the reconstruction of the boundary to be less affected by the loss of ellipticity. This is indeed also what is seen in Fig.~\ref{fig:animaffls}, where the propagation of singularities happens in different directions. Again, on the top of the figures we have placed arrows which show the predicted directions in which singularities propagate.  Notice how the boundary of the perturbation is not captured well by the reconstruction, when the loss of ellipticity aligns with the direction in which the boundary is singular.
\begin{figure}[!htb]
        \centering
        \begin{subfigure}[b]{0.22\textwidth}
                \centering
               \includegraphics[width=\textwidth]{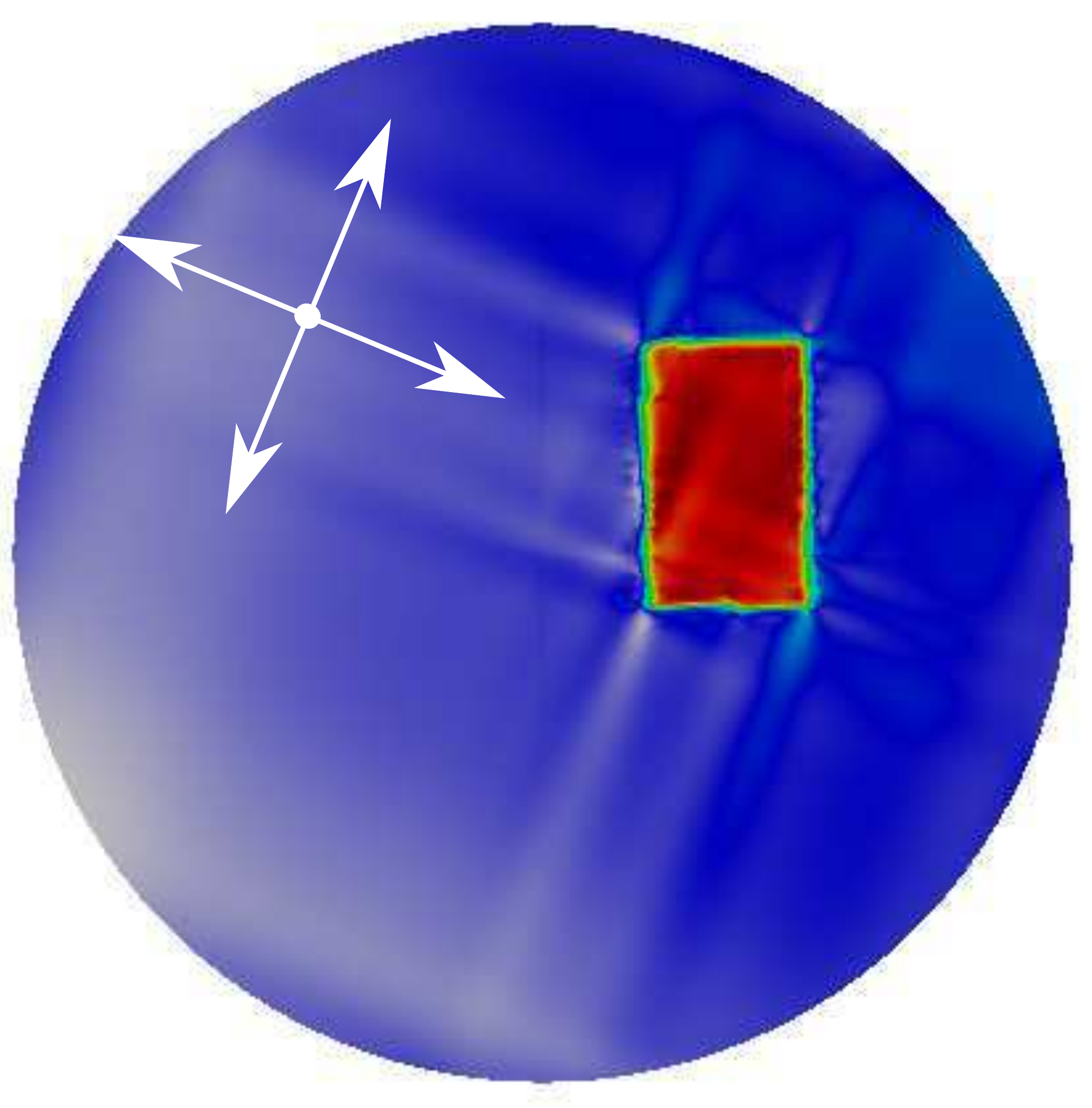}    
                \label{fig:mdsfouse}
        \end{subfigure}
           \begin{subfigure}[b]{0.22\textwidth}
                \centering
    \includegraphics[width=\textwidth]{nep2f12-eps-converted-to.pdf}  
                \label{fig:mou345se}
        \end{subfigure}
                          \begin{subfigure}[b]{0.22\textwidth}
                \centering
              \includegraphics[width=\textwidth]{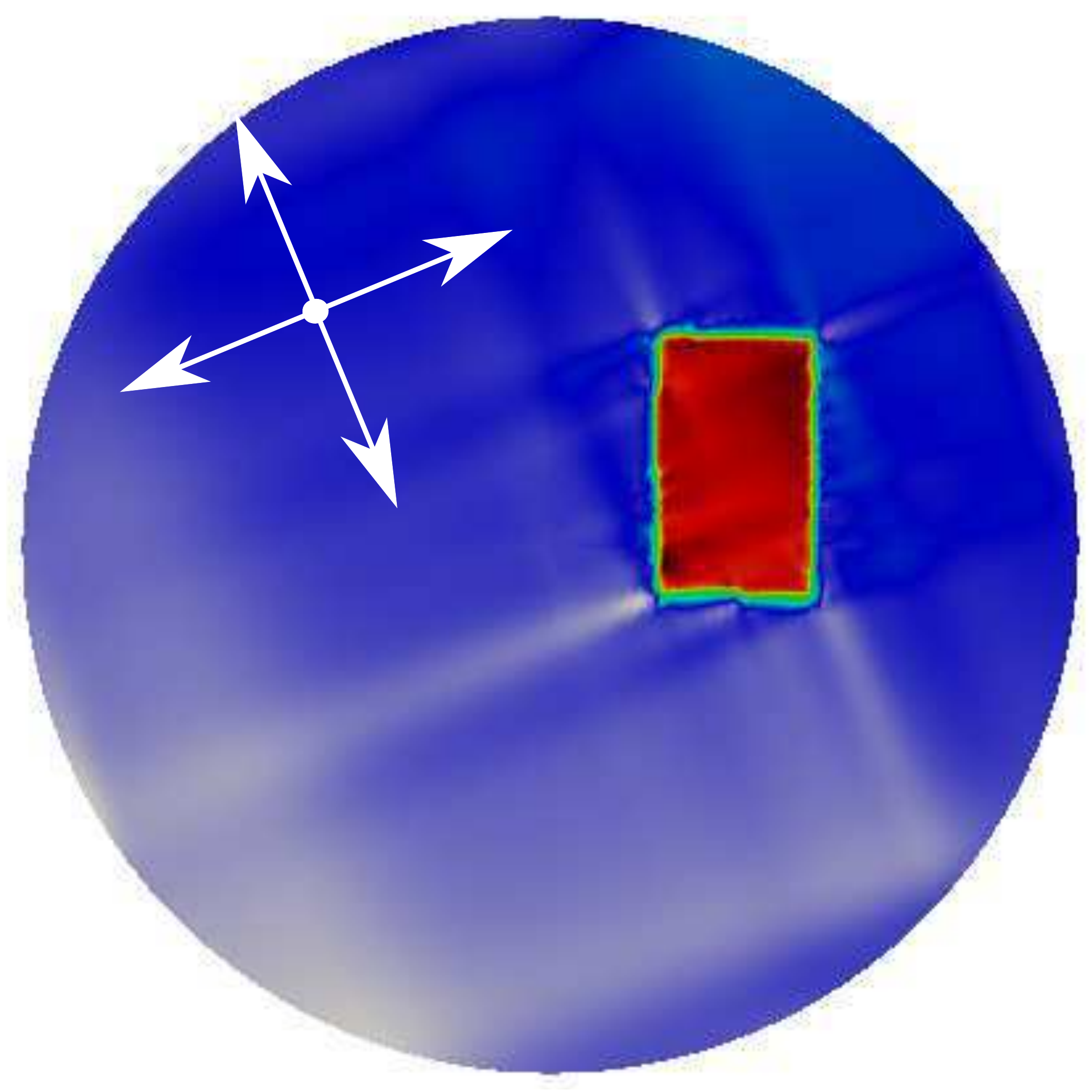}    
                \label{fig:mosfuaaaase}
        \end{subfigure}
                                  \begin{subfigure}[b]{0.22\textwidth}
                \centering
              \includegraphics[width=\textwidth]{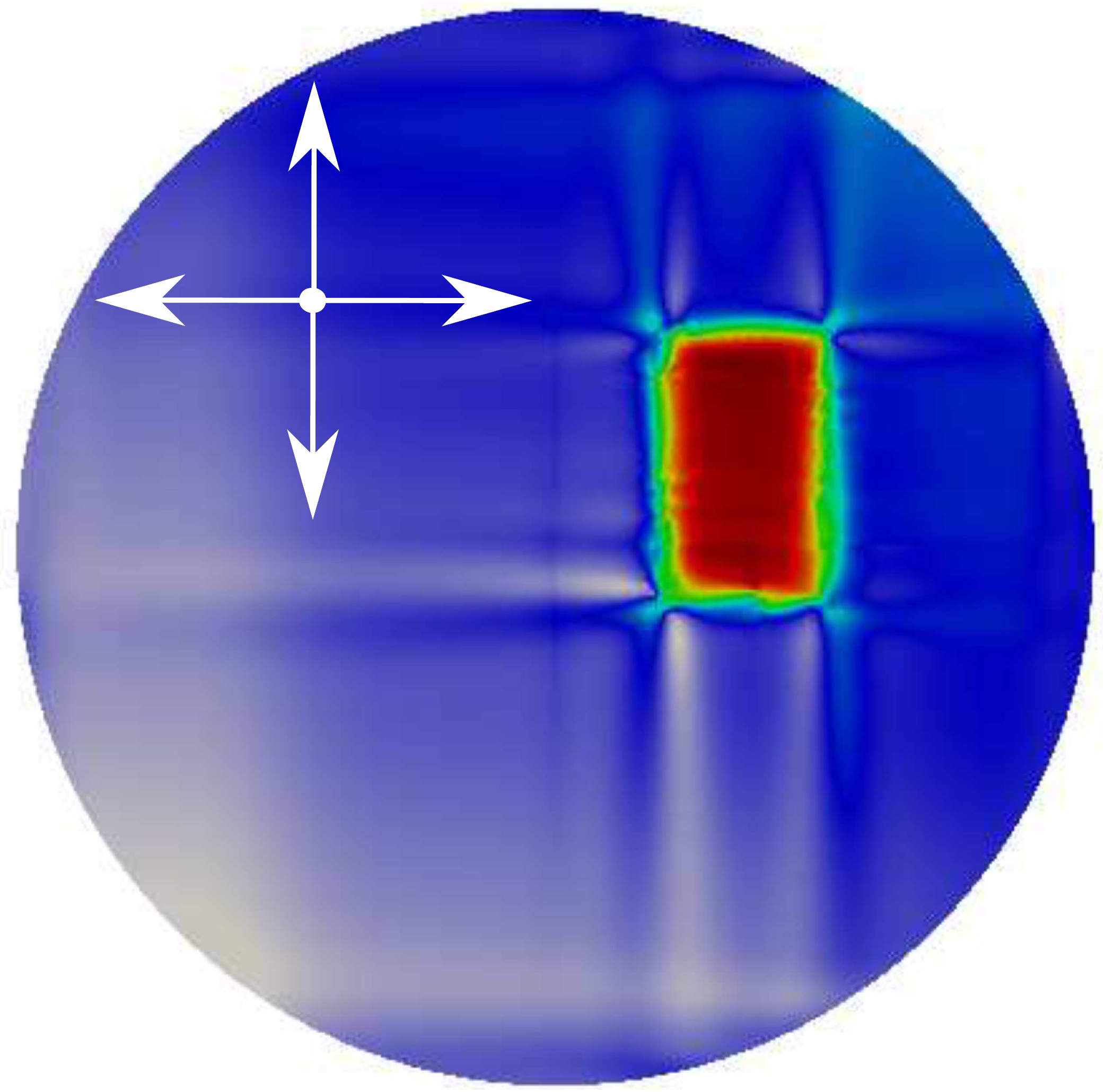}  
                \label{fig:mosfuse}
        \end{subfigure}
         \begin{subfigure}[b]{0.038\textwidth}
                \centering
                \includegraphics[width=\textwidth]{colorbar-eps-converted-to.pdf}
                
             \caption*{}
        \end{subfigure}
        ~ 
       \caption{The propagation of singularities can align with the singularities in the data. Here it is presented for $p=2$. If the propagation happens along curves where the phantom itself is singular it can dramatically affect the reconstruction of the boundary of the perturbation, as seen on the picture to the right.}\label{fig:animaffls}
\end{figure}

\section{Conclusion}
The similarity between an exact solution to the original non-linear inverse problem and the solution obtained from the linearised inverse problem, is highly dependent on the presence of non-smooth errors in the reconstruction. It is clear that any type of characterisation of such errors is very valuable. A theoretical analysis based on the theory of propagation of singularities provided results on when singularities will be present in the reconstruction and how these are related to the singularities in the data. In the case $p>1$, it was explained in Chap. 3 how the linearised problem and the corresponding normal problem, permit propagation of singularities in directions perpendicular to any direction along which ellipticity is lost. The linearised inverse problem was implemented numerically by means of the least squares finite element method. The corresponding weak formulation turned out to be simple and is formulated independently on the number of measurements. 
To verify the theoretical results a numerical analysis was performed for a simple phantom and the numerical implementation produced reconstructions with features that could be explained by the derived theoretical results.


\bibliographystyle{plain}
\bibliography{bib}	

\begin{thebibliography}{10}

\bibitem{alessandrini2001univalent}
Giovanni Alessandrini and Vincenzo Nesi.
\newblock Univalent $\sigma$-harmonic mappings.
\newblock {\em Archive for Rational Mechanics and Analysis}, 158(2):155--171,
  2001.

\bibitem{ammari2008introduction}
Habib Ammari.
\newblock {\em An introduction to mathematics of emerging biomedical imaging},
  volume~62 of {\em Math{\'e}matiques \& applications}.
\newblock Springer, 2008.

\bibitem{bal2012hybridreview}
Guillaume Bal.
\newblock {\em Hybrid inverse problems and internal functionals}.
\newblock Cambridge University Press, Cambridge, UK, G. Uhlmann, Editor, 2012.

\bibitem{bal2012cauchy}
Guillaume Bal.
\newblock {Cauchy problem for Ultrasound Modulated EIT}.
\newblock {\em {Analysis and PDE}}, 6(4):751--775, 2013.

\bibitem{bal2012hybrid}
Guillaume Bal.
\newblock Hybrid inverse problems and redundant systems of partial differential
  equations.
\newblock In {\em Inverse problems and applications}, volume 615 of {\em
  Contemp. Math.}, pages 15--47. Amer. Math. Soc., Providence, RI, 2014.

\bibitem{bochev2009least-squares}
Pavel Bochev.
\newblock {\em Least-squares finite element methods}.
\newblock Springer, New York, 2009.

\bibitem{chazarain1974propagation}
Jacques Chazarain.
\newblock Propagation des singularit{\'e}s pour une classe d'op{\'e}rateurs
  {\`a} caract{\'e}ristiques multiples et r{\'e}solubilit{\'e} locale.
\newblock In {\em Annales de l'institut Fourier}, volume~24, pages 203--223.
  Institut Fourier, 1974.

\bibitem{chazarain1976reflection}
Jacques Chazarain.
\newblock Reflection of ${C}^\infty$ singularities for a class of operators
  with multiple characteristics.
\newblock {\em Publications of the Research Institute for Mathematical
  Sciences}, 12(99):39--52, 1976.

\bibitem{duistermaat1972fourier}
Johannes~Jisse Duistermaat and Lars H{\"o}rmander.
\newblock Fourier integral operators. {II}.
\newblock {\em Acta mathematica}, 128(1):183--269, 1972.

\bibitem{egorov1993microlocal}
Yu.V. Egorov and M.A. Shubin.
\newblock Microlocal analysis.
\newblock In {\em Partial Differential Equations IV}, Encyclopaedia of
  Mathematical Sciences, pages 1--147. Springer, 1993.

\bibitem{evans2010partial}
Lawrence~C. Evans.
\newblock {\em Partial differential equations}, volume~19 of {\em Graduate
  Studies in Mathematics}.
\newblock American Mathematical Society, Providence, RI, second edition, 2010.

\bibitem{Frikel2013}
J\"urgen Frikel.
\newblock Sparse regularization in limited angle tomography.
\newblock {\em Appl. Comput. Harmon. Anal.}, 34(1):117--141, 2013.

\bibitem{FrikelQuinto2013}
J\"urgen Frikel and Eric~Todd Quinto.
\newblock Characterization and reduction of artifacts in limited angle
  tomography.
\newblock {\em Inverse Problems}, 29(12):125007, 21, 2013.

\bibitem{FrikelQuinto2015}
J\"urgen Frikel and Eric~Todd Quinto.
\newblock Artifacts in incomplete data tomography with applications to
  photoacoustic tomography and sonar.
\newblock {\em SIAM J. Appl. Math.}, 75(2):703--725, 2015.

\bibitem{gilbarg1977elliptic}
David Gilbarg and Neil~S. Trudinger.
\newblock {\em Elliptic Partial Differential Equations of Second Order}.
\newblock Die Grundlehren der mathematischen Wissenschaften in
  Einzeldarstellungen. Springer-Verlag, 1977.

\bibitem{henderson2004the}
Amy Henderson.
\newblock {\em The ParaView guide}.
\newblock Kitware, Clifton Park, NY, 2004.

\bibitem{komech1999elements}
A.I. Komech and A.~Shubin.
\newblock {\em Elements of the Modern Theory of Partial Differential
  Equations}.
\newblock Springer Berlin Heidelberg, 1999.

\bibitem{kuchment2012mathematics}
Peter Kuchment.
\newblock Mathematics of hybrid imaging: A brief review.
\newblock In {\em The Mathematical Legacy of Leon Ehrenpreis}, pages 183--208.
  Springer, 2012.

\bibitem{kuchment2012stabilizing}
Peter Kuchment and Dustin Steinhauer.
\newblock Stabilizing inverse problems by internal data.
\newblock {\em Inverse Problems}, 28(8):084007, 2012.

\bibitem{logg2012automated}
Anders Logg.
\newblock {\em Automated solution of differential equations by the finite
  element method the FEniCS book}.
\newblock Springer, Berlin New York, 2012.

\bibitem{nachman2007conductivity}
Adrian Nachman, Alexandru Tamasan, and Alexandre Timonov.
\newblock Conductivity imaging with a single measurement of boundary and
  interior data.
\newblock {\em Inverse Problems}, 23(6):2551, 2007.

\bibitem{reed1975methods}
Michael Reed and Barry Simon.
\newblock {\em Fourier Analysis, Self-Adjointness}.
\newblock Number~II in Methods of Modern Mathematical Physics. Academic Press,
  1975.

\bibitem{shubin2001pseudodifferential}
A.~Shubin.
\newblock {\em Pseudodifferential Operators and Spectral Theory}.
\newblock Pseudodifferential Operators and Spectral Theory. Springer Berlin
  Heidelberg, 2001.

\bibitem{tartar2007introduction}
Luc Tartar.
\newblock {\em An Introduction to Sobolev Spaces and Interpolation Spaces}.
\newblock Lecture notes of the Unione Matematica Italiana. Springer-Verlag
  Berlin Heidelberg, 2007.

\bibitem{Treves1}
Jean-François Treves.
\newblock {\em Introduction to Pseudodifferential and Fourier Integral
  Operators Volume 1: Pseudodifferential Operators (University Series in
  Mathematics)}.
\newblock Springer, 1980.

\end{thebibliography}
\end{document}